\documentclass{amsart}

\usepackage{amsmath}

\usepackage[textsize=tiny]{todonotes}

\usepackage{xcolor}

\usepackage{amssymb}
\usepackage[initials]{amsrefs}
\usepackage{comment}
\usepackage{bbm}
\usepackage{tikz}
\usetikzlibrary{matrix}
\usepackage[all,cmtip]{xy}
\usepackage{hyperref}
\usepackage{stmaryrd}
\usepackage{soul}

\usepackage{textcomp}

\newtheoremstyle{mesthm}
{10pt plus 1pt minus 1pt}
{9pt minus 6pt}
{\slshape}
{0.5cm}
{\bfseries}
{.}
{1ex}
{}
\theoremstyle{mesthm}

\newcounter{mt}

\newtheorem{Proposition}{Proposition}[section]
\newtheorem{Definition}[Proposition]{Definition}
\newtheorem{Lemma}[Proposition]{Lemma}
\newtheorem{Theorem}[Proposition]{Theorem}
\newtheorem{Corollary}[Proposition]{Corollary}

\theoremstyle{remark}
\newtheorem{Remark}[Proposition]{Remark}

\DeclareMathOperator{\Val}{Val}

\DeclareMathOperator{\nc}{nc}

\DeclareMathOperator{\Gr}{Gr}

\DeclareMathOperator{\sgn}{sgn}
\DeclareMathOperator{\id}{id}

\DeclareMathOperator{\vol}{vol}

\DeclareMathOperator{\Dens}{Dens}
\DeclareMathOperator{\Span}{Span}

\DeclareMathOperator{\Sym}{Sym}
\DeclareMathOperator{\Ker}{ker}

\DeclareMathOperator{\ori}{or}

\DeclareMathOperator{\GL}{GL}

\DeclareMathOperator{\sign}{sign}

\newcommand{\R}{\mathbb{R}}
\newcommand{\C}{\mathbb{C}}
\newcommand{\RR}{\mathbb{R}}
\newcommand{\CC}{\mathbb{C}}
\newcommand{\PP}{\mathbb{P}}

\newcommand{\largewedge}{\mathbb{\wedge}}

\newcommand{\FF}{\mathbb{F}}
\newcommand{\FC}{\mathcal{F}}
\newcommand{\wt}{\widetilde}
\newcommand{\calK}{\mathcal{K}}
\newcommand{\dens}{\mathrm{Dens}}
\newcommand{\calS}{\mathcal{S}}

\makeatletter
\DeclareFontFamily{OMX}{MnSymbolE}{}
\DeclareSymbolFont{MnLargeSymbols}{OMX}{MnSymbolE}{m}{n}
\SetSymbolFont{MnLargeSymbols}{bold}{OMX}{MnSymbolE}{b}{n}
\DeclareFontShape{OMX}{MnSymbolE}{m}{n}{
	<-6>  MnSymbolE5
	<6-7>  MnSymbolE6
	<7-8>  MnSymbolE7
	<8-9>  MnSymbolE8
	<9-10> MnSymbolE9
	<10-12> MnSymbolE10
	<12->   MnSymbolE12
}{}
\DeclareFontShape{OMX}{MnSymbolE}{b}{n}{
	<-6>  MnSymbolE-Bold5
	<6-7>  MnSymbolE-Bold6
	<7-8>  MnSymbolE-Bold7
	<8-9>  MnSymbolE-Bold8
	<9-10> MnSymbolE-Bold9
	<10-12> MnSymbolE-Bold10
	<12->   MnSymbolE-Bold12
}{}

\let\llangle\@undefined
\let\rrangle\@undefined
\DeclareMathDelimiter{\llangle}{\mathopen}%
{MnLargeSymbols}{'164}{MnLargeSymbols}{'164}
\DeclareMathDelimiter{\rrangle}{\mathclose}%
{MnLargeSymbols}{'171}{MnLargeSymbols}{'171}
\makeatother

\BibSpec{article}{%
	+{}{\PrintAuthors}  		{author}
	+{,}{ \textit}     		{title}
	+{,}{ }             		{journal}
	+{}{ \textbf}       		{volume}
	+{}{ \parenthesize} 		{date}
	+{}{, no. } 		{number}
	+{,}{ }      	      		{conference}
	+{,}{ }      	      		{book}
	+{,}{ }            		{pages}
	+{,}{ }            	 	{note}
	+{,}{ }            	 	{status}
	+{,}{  \texttt } {eprint}
	+{.}{}              {transition}
}

\BibSpec{book}{%
	+{}{\PrintAuthors}  {author}
	+{,}{ \textit}      {title}
	+{,}{ }             {publisher}
	+{,}{ }             {place}
	+{,}{ }             {date}
	+{.}{}              {transition}
}

\BibSpec{incollection}{%
	+{}  {\PrintAuthors}                {author}
	+{,} { \textit}                     {title}
	+{.} { }                            {part}
	+{:} { \textit}                     {subtitle}
	+{,} { \PrintContributions}         {contribution}
	+{,} { \PrintConference}            {conference}
	+{,} { }                            {booktitle}
	+{,} { pp.~}                        {pages}
	+{,}{ }       		{series}
	+{}{ \textbf}       		{volume}
	+{,} { }                            {publisher}
	+{,} { }                 {date}
	+{,} { }                            {status}
	+{,} { \PrintDOI}                   {doi}
	+{,} { available at \eprint}        {eprint}
	+{}  { \parenthesize}               {language}
	+{}  { \PrintTranslation}           {translation}
	+{;} { \PrintReprint}               {reprint}
	+{.} { }                            {note}
	+{.} {}                             {transition}
}

\BibSpec{inproceedings}{%
	+{}  {\PrintAuthors}                {author}
	+{,} { \textit}                     {title}
	+{.} { }                            {part}
	+{:} { \textit}                     {subtitle}
	+{,} { \PrintContributions}         {contribution}
	+{,} { \PrintConference}            {conference}
	+{,} { }                            {booktitle}
	+{,} { pp.~}                        {pages}
	+{,}{ }       		{series}
	+{}{ \textbf}       		{volume}
	+{,} { }                            {publisher}
	+{,} { }                 {date}
	+{,} { }                            {status}
	+{,} { \PrintDOI}                   {doi}
	+{,} { available at \eprint}        {eprint}
	+{}  { \parenthesize}               {language}
	+{}  { \PrintTranslation}           {translation}
	+{;} { \PrintReprint}               {reprint}
	+{.} { }                            {note}
	+{.} {}                             {transition}
}

\title[]{
The Fourier transform on valuations  is the Fourier transform
}
\author{Dmitry Faifman}
\address{School of Mathematical Sciences, Tel Aviv University, Tel Aviv 6997801, Israel}
\email{faifmand@tauex.tau.ac.il}

\author{Thomas Wannerer}
\address{Friedrich-Schiller-Universitat Jena, Fakult\"at f\"ur Mathematik und Informatik, Institut f\"ur Mathematik, Ernst-Abbe-Platz 2, 07743 Jena, Germany}
\email{thomas.wannerer@uni-jena.de}

\thanks{DF was supported by the Israel Science Foundation grant No. 1750/20.} \thanks{TW was supported by DFG grant WA 3510/3-1.}
\date{\today}
\subjclass[2020]{52B45, 53C65, 42B10, 43A32}

\begin{document}
\begin{abstract}Alesker has proved the existence of a remarkable isomorphism of the space of translation-invariant smooth valuations that has the same functorial properties as the classical Fourier transform. In this paper, we  show how to directly describe  this isomorphism in terms of the Fourier transform on functions. As a consequence, we obtain simple proofs of the main properties of the Alesker--Fourier transform. One of these properties was previously only conjectured by Alesker and is proved here for the first time.
\end{abstract}

\maketitle

\sloppy
\section{Introduction}

\subsection{The Alesker--Fourier transform}

Let $V$ be an $n$-dimensional real vector space and let $\calK(V)$ denote the space of convex bodies, i.e., non-empty convex compact subsets of $V$. 
In convex geometry, a valuation   is a function $\phi\colon \calK(V)\to \CC$ satisfying the property
$$ \phi(K\cup L) = \phi(K) + \phi(L) - \phi(K\cap L)$$
whenever the union of $K$ and $L$ is convex.  Valuations, in particular in connection with dissection problems for polytopes and integral geometry,  are a classical part of convex geometry, see the books by Schneider~ \cite{Schneider:BM} and Gruber~\cite{Gruber:CDG} for more information. 

About two decades ago, the revolutionary work of Alesker uncovered a 
surprisingly rich algebraic structure of the space of translation-invariant and continuous valuations. The latter space, denoted by $\Val(V)$, has been intensively studied through the work of  Hadwiger \cite{Hadwiger:Vorlesungen}, McMullen \cite{McMullen:Euler, McMullen:Continuous}, Schneider \cite{Schneider:Simple}, Goodey--Weil \cite{GoodeyWeil:Distributions}, Klain \cite{Klain:Short, Klain:Even}, and others \cite[Chapter 6]{Schneider:BM}. One of Alesker's key insights was to identify  a natural dense subspace of smooth valuations $\Val^\infty(V)\subset \Val(V)$ with remarkable properties. The most important one is the existence of a product, now called the Alesker product, that gives this space  the structure of commutative graded algebra with identity \cite{Alesker:Product}. Working with smooth valuations is crucial, as the Alesker product and most of the other operations do not continuously extend to the larger space $\Val(V)$. 
Building on Alesker's work, Bernig and Fu \cite{BernigFu:Convolution} discovered a second multiplicative structure called the convolution of valuations. 

Injective linear maps $f\colon V\hookrightarrow W$ induce a pullback  of valuations
$$f^*\colon \Val^\infty(W)\to \Val^\infty(V),$$ 
which respects the Alesker product, while surjective linear maps induce a pushforward 
$$ f_* \colon \Val^\infty(V)\otimes \dens(V^*)\to \Val^\infty(W)\otimes \dens(W^*),$$
which respects the Bernig--Fu convolution, see \cite{Alesker:Fourier}. Here $\dens(V)$ denotes the one-dimensional space of densities on $V$, see below for a precise definition. Product and convolution admit a common description in terms of the exterior product of valuations 
$$ \Val^\infty(V)\times \Val^\infty(W)\to \Val(V\times W)$$
denoted by $\phi\boxtimes \psi$. Namely, using the operations of pullback and pushforward, 
$$ \phi \cdot \psi = \Delta^*(\phi\boxtimes \psi) \quad \text{and} \quad \phi * \psi = a_*(\phi\boxtimes \psi),$$
where $\Delta\colon V\to V\times V$ is the diagonal embedding, and $a\colon V\times V\to V$ is the addition in the vector space $V$.

The Alesker--Fourier transform,  or simply the Fourier transform on valuations enriches this picture even further. First defined by Alesker for even valuations in \cite{Alesker:HL} and only later in \cite{Alesker:Fourier} constructed in full generality, it is an isomorphism
\[\FF: \Val^\infty(V)\to \Val^\infty(V^*)\otimes\Dens(V)\]
which interchanges the  product with convolution, pullback with pushforward, and satisfies the inversion formula: applying the Alesker--Fourier transform twice equals  the pullback by the antipodal map $x\mapsto -x$.  Its construction in the general case rested on Alesker's highly non-trivial irreducibility theorem, as well as sophisticated methods from infinite-dimensional representation theory of $\GL_n(\R)$.

The name ``Fourier'' was attached to $\FF$ due to its functorial properties, which are strongly reminiscent of the classical Fourier transform on functions.  
More recently however, a special case of the Alesker--Fourier transform was observed to be indeed linked to the Fourier transform of functions \cite[Corollary 3.9]{DorrekSchuster:AF}.

\subsection{Our results}

In this work, we give a description of the Alesker--Fourier transform that directly derives it from the Fourier transform on functions, or rather differential forms, in all cases, and directly deduces its key properties from the corresponding properties of the Fourier transform. This solves a problem posed by Alesker \cite[p. 17]{AleskerFu:Barcelona}. Moreover, our approach allows us to establish an additional property of the Fourier transform, which were conjectured by Alesker \cite[p. 17]{AleskerFu:Barcelona}.

 In particular, we make no use of the irreducibility theorem of Alesker \cite{Alesker:Irreducibility}, assuming only the definition of smooth valuations through integration over the normal cycle: A valuation $\phi\in \Val(V)$ is smooth if  there exists a smooth differential form $\omega$ on $\PP_+(V^*)$, the oriented projectivization of $V^*$, with values in $\largewedge^\bullet V^* \otimes \mathrm{or}(V)$, and a density $\theta$ on $V$ such that 
\begin{equation}\label{eq:smoothVal} \phi(K)= \theta(K) + \int_{\nc(K)} \omega\end{equation}
for all convex bodies $K$ in $V$. Here $\nc (K)$ denotes the normal cycle of $K$, and $\ori(V)$ denotes the one-dimensional space of orientation functions on $V$.

In the first step of our construction of the Fourier transform on valuations, which seems to be also of independent interest, we associate to  every  smooth valuation $\phi\in \Val^\infty(V)$ a unique $0$-homogeneous  current $\tau=\tau(\phi)$ on $V^*$, which is smooth outside of the origin. More precisely, if $\phi \in \Val_k^\infty(V)$ is given by \eqref{eq:smoothVal}, then we define a generalized form $\Omega_{-\infty}^{ n-k}( V^*, \largewedge^{ k} V^* \otimes \mathrm{or}(V))$ by
$$\tau(\phi)=  \phi(\{0\}) \cdot \delta_0 +  r^0( (-1)^{n-k} a^*D\omega +  \theta).$$
Here $D$ is the Rumin differential of $\omega$, a natural second order differential operator arising in contact geometry, $a$ is the antipodal map, and $r^0$ denotes the $0$-homogeneous extension from $\PP_+(V^*)$ to $V^*$. For more details, see section \ref{sec:valuation_current}. A closely related description of valuations in terms of a pair of currents was introduced in \cite{AleskerBernig:Product} building on earlier work of Bernig and Br\"ocker \cite{BernigBroecker:Rumin}.

The main point of this construction is that we are able to characterize all generalized forms arising in this fashion by a short list of inevitable properties (Proposition~\ref{prop:forms_121_valuations}). 
Moreover, the  operations of pullback, pushforward, and exterior product of valuations admit simple descriptions in terms of the $0$-homogeneous current  $\tau(\phi)$, see Propositions~\ref{prop:pullback_linear}, \ref{prop:pushforward_linear}, and \ref{prop:exterior_product}. Furthermore, the $0$-homogeneous current of a valuation allows an explicit description of the $0$-homogeneous component of the pushforward of a valuation by an epimorphism, which is not readily available when using the standard presentation by a pair of currents.

Let $F$ be  a finite-dimensional vector spaces of over the reals, and 
let $\Omega_{\mathcal S}^k(V, F)\subset \Omega^k(V,F)$ denote the Schwartz space of differential $k$-forms on $V$ with values in $F$ and rapidly decreasing coefficients. The Fourier transform of functions on $V$ extends naturally to differential forms on $V$. 

\begin{Definition}
	The Fourier  transform $$\FC: \Omega_{\mathcal S}^k(V, F) \to \Omega_{\mathcal S}^{n-k}(V^*, \ori(V)\otimes F)$$  is defined as follows.  For $\omega \in \Omega_{\mathcal S}^k(V, F)$, 	
	the Hodge star isomorphism 
	$$\ast: \largewedge^k V^* \xrightarrow{\sim} \largewedge^{n-k} V \otimes \largewedge^n V^* \simeq  \largewedge^{n-k} V \otimes \Dens(V) \otimes \ori(V)$$
	allows to consider $\omega$ as a map $\wt \omega:V\to  \largewedge^{n-k} V \otimes \Dens(V) \otimes \ori(V)\otimes F$.
	Hence we may define for $\xi\in V^*$
	$$ \FC(\omega)(\xi) =\int_{V} e^{2\pi \mathbf i \langle x,\xi\rangle}  \wt\omega(x) \in   \largewedge^{n-k} V\otimes \ori(V)\otimes F.$$
\end{Definition}
We remark that while different normalizations exist for the Fourier transform, an involutive Fourier transform on $0$-homogeneous forms is unique up to sign.

We extend this definition also to differential forms with tempered distributions as coefficients. 
In applications to  valuations,  the vector space $F$ is $\largewedge^\bullet V^*\otimes \mathrm{or}(V)$. If we denote the composition of $\FC$ with the canonical isomorphism 
$$F \otimes \mathrm{or}(V) \simeq \largewedge ^\bullet V \otimes \dens(V)\otimes \ori (V),$$ 
by $\FC^0$, then
$$ \FC^0\tau(\phi)\in \Omega_{\mathcal S'} (V, \largewedge^\bullet V \otimes  \ori(V)) \otimes \dens(V).$$  
Our first theorem states that this form defines a smooth valuation.

\begin{Theorem} Let $\phi\in\Val^\infty(V)$ be a smooth valuation, and let $\tau(\phi)$ be its $0$-homogeneous current. Then the Fourier transform 
	$ \FC^0 \tau(\phi)$ is the $0$-homogeneous current of a unique smooth valuation. 
\end{Theorem}

This theorem  allows us to make the following definition. 

\begin{Definition} \label{def:FT}
	If  $\phi\in\Val^\infty(V)$ is a smooth valuation, then its $\FC$-transform  is the unique smooth  valuation
	$\FC\phi\in \Val^\infty(V^*)\otimes \dens(V)$ satisfying 
	$$  \tau(\FC \phi)= \FC^0\tau(\phi).$$ 
\end{Definition}

The same equation also defines the  $\FC$-transform of generalized valuations $\phi\in \Val^{-\infty}(V)$.  It will be seen in Theorem \ref{thm:F=F} that the $\FC$-transform coincides with the Alesker--Fourier transform.

With this definition  and the description of the algebraic operations of valuations in terms of the $0$-homogeneous current, the main properties of the Alesker--Fourier transform are for the $\FC$-transform a direct consequence of the corresponding properties of the Fourier transform of functions. 
We write $\FC_V$ for the $\FC$-transform of valuations on $V$  in case we consider several vector spaces at the same time.

\begin{Theorem} \label{thm:properties} The $\FC$-transform of valuations has the following properties:
	\begin{enumerate}
		\item $\FC$ commutes with the natural action of $\mathrm{GL}(V)$.
		\item Inversion formula: $(\FC_{V^*}\times \id)\circ \FC_V\phi =(-\id)^*\phi$.
		\item  \label{F:injection} Let $i:V\hookrightarrow W$ be an injective linear map, and $p=i^\vee:W^*\to V^*$. Then for $\phi\in \Val^\infty(W)$, $$\FC_V( i^* \phi) = p_* (\FC_W\phi).$$
		\item \label{F:extProd}$\FC (\phi \boxtimes \psi)= \FC\phi \boxtimes \FC\psi$ for  $\phi\in \Val^\infty(V)$, $\psi\in \Val^{\infty}(W)$.
		\item $\FC$ intertwines product and convolution: for $\phi,\psi \in\Val^\infty(V)$, $$\FC (\phi\cdot \psi)=\FC\phi\ast\FC\psi.$$ 
		\item  \label{F:sa} Self-adjointness: $\langle \FC u,  \theta \rangle = \langle u, \FC \theta\rangle  $ for $u \in \Val^{-\infty}(V)$, $\theta\in \Val^\infty(V^*)$.
		\item  Let $p:V\to W$ be a surjective linear map, and $i=p^\vee:W^*\to V^*$. Then for $\psi\in\Val^{-\infty}(W)$, $$\FC_V  ( p^* \psi)= i_*(\FC_W\psi).$$  
	\end{enumerate}
\end{Theorem}

Using the above properties of the $\FC$-transform we finally prove
\begin{Theorem} \label{thm:F=F}
	The $\FC$-transform and the Alesker--Fourier transform coincide: $\FC= \FF$. 
\end{Theorem}

 Taken together, the properties \eqref{F:injection} and \eqref{F:surjection}  imply the identity 
$$\FC \circ f^* = (f^\vee)_* \circ \FC$$
 on smooth valuations for every linear map $f$ with dual map $f^\vee$. 
 
The  self-adjointness \eqref{F:sa} of the Alesker--Fourier transform was  already established in the PhD thesis of the first-named author \cite{faifman_thesis}, using Alesker's original definition, and employed to extend the Alesker--Fourier transform to the space of generalized valuations.  
This extension  gives formal meaning to properties \eqref{F:extProd}  and \eqref{F:surjection}, which  were previously conjectured by Alesker~\cite[p. 17]{AleskerFu:Barcelona}. While properties \eqref{F:injection} and \eqref{F:surjection} are easily seen to be equivalent   via self-adjointness and the inversion formula, property \eqref{F:extProd} is  proved here for the first time.

\subsection{Relation to other work}
The relevance  of the classical Fourier transform for questions about volumes of  sections and projections of convex bodies was realized in 1990s,  when it led to a complete  solution of the Busemann--Petty problem.  The application of Fourier analytic techniques to problems in convex geometry has developed into a rich area of research, see the papers \cite{Ball:Cube,Gardner_etal:AnalyticSolution,Koldobsky:Intersection,Koldobsky_etal:Projections, fourier_santalo, nazarov_mahler}, the lecture notes by Ryabogin and Zvavitch \cite{RyaboginZvavitch:Methods}, and Koldobsky's monograph \cite{Koldobsky:FourierAnalysis} for more information.

As already mentioned above, that the Alesker--Fourier transform can be linked to the Fourier transform on functions was first  observed by Dorrek and Schuster \cite{DorrekSchuster:AF}. 
This special case motivated us to look for such a connection in general.

The Alesker--Fourier transform has important applications  in integral geometry.    Kinematic formulas are a centerpiece of integral geometry, first studied for the orthogonal Euclidean group by Blaschke,  Chern, and Santal\'o, see the books by  Santal\'o~\cite{Santalo} and Klain and Rota~\cite{KlainRota}. For convex bodies $A,B$ in $\RR^n$ they come in two flavors, namely intersectional
$$ \int_{G\ltimes \RR^n} \chi(A \cap g B) \,dg = \sum_{i,j} c_{ij} \mu_i(A) \mu_j(B)$$ 
and additive 
$$\int_{G} \vol(A+ gB)\, dg=\sum_{i,j} d_{ij} \mu_i(A) \mu_j(B),$$ 
where  $G$ is a compact Lie group acting transitively on the unit sphere, integration is with respect to the Haar measure, the $\mu_i$ are a basis of $\Val^G$, the space of $G$-invariant valuations, and the constants $c_{ij},d_{ij}$ are independent of $A$ and $B$. More generally, the Euler characteristic and volume may be replaced by any of the $\mu_i$. 

Explicitly determining the constants in the  kinematic formulas is a challenge.  A key  insight due to Fu \cite{Fu:Unitary} is that they are the structure constants of the algebras of $G$-invariant valuations for product and convolution. 
Dually to the  Alesker--Fourier transform intertwining product and convolution, it intertwines  intersectional and  additive kinematic formulas. Since the  convolution is typically simpler to compute than the Alesker product, this allows  at least in principle  to deduce both types of kinematic formulas from the convolution table of the invariant valuations.

 For the unitary groups  $\mathrm U(n)$ and $\mathrm {SU}(n)$, this program has been successfully completed  by Bernig and Fu \cite{BernigFu:Hig, Bernig:SU}. While determining the  Alesker--Fourier transform is relatively easy in these cases, in the works of Bernig and Solanes \cite{BernigSolanes:ClassificationH2, BernigSolanes:IGH2} on the integral geometry of the quaternionic plane and Bernig and Hug \cite{BernigHug:Tensor} on the integral geometry of tensor valuations, this step relies on a sizable and difficult argument. We have not attempted to do this, but it seems plausible  that our explicit description of the  Alesker--Fourier transform of valuations could simplify these arguments.

The Alesker--Fourier transform can also be used to construct the  Holmes--Thompson intrinsic volumes in a normed space. Crofton formulas in normed spaces were first constructed Schneider and Wieacker~\cite{SchneiderWieacker:CroftonFinsler} and  Alvarez-Paiva and Fernandes~\cite{APF:CroftonFinsler}, which were then used by Bernig~\cite{Bernig:CroftonFinsler} to construct a natural family of valuations on all normed spaces, now called intrinsic volumes, extending the Holmes--Thompson volume on all flats. As observed by Fu~\cite[Theorem 2.3.22]{Fu:Barcelona}, the simplest definition of those valuations is in fact as the (inverse)  Alesker--Fourier transform of the mixed volume with the polar of the unit ball, explicitly
\[ V^{\mathrm{HT}}_k=  \frac{1}{\omega_{k}} \binom{n}{k}  \FF^{-1} V(B^\circ[k], \bullet[n-k]),\]
where $\omega_k$ denotes the volume of the $k$-dimensional Euclidean unit ball. 
 Moreover, for non-symmetric normed spaces the latter is the only known way to define the intrinsic volumes so that they are generated from the first one under the Alesker product, see \cite{WannererFaifman:Finsler}. 

Let us finally point out that the Alesker--Fourier transform was used to deduce  new inequalities for mixed volumes of  convex bodies from the recently established Hodge--Riemann relations on the space of smooth valuations \cite{Alesker:Kotrbaty, Kotrbaty:HR, KotrbatyWannerer:AF}.

\subsection{Acknowledgements}
We are grateful to Semyon Alesker for his valuable comments on an earlier draft.

\section{Preliminaries}

If $V$ is an  $n$-dimensional real vector space, we denote by $\largewedge^\bullet V = \bigoplus_{k=0}^n \largewedge^k V$ the exterior algebra of $V$. A density on $V$ is a function $\mu\colon \largewedge^n V \to \CC $ with the property that $\mu(t \omega)=  |t| \mu(\omega)$ holds for all $t\in \RR$. The $1$-dimensional space of densities is denoted by $ \Dens(V)$,  and its nonzero elements are naturally identified with the complex-valued Lebesgue measures on $V$. An orientation function on $V$ is a function $\epsilon\colon \largewedge^n V\setminus\{0\} \to \CC $ satisfying
 $\epsilon(t \omega)=  \sgn(t) \epsilon(\omega)$ for all $t\in \RR\setminus \{0\}$.
The $1$-dimensional space of  orientation functions is denoted by $\ori(V)$. We will make use of the following canonical isomorphisms 
\begin{align*}
	\Dens(V)\otimes \ori(V) &\simeq \largewedge^n V^*, \\
	\ori(V) &\simeq \ori(V^*),\\
	\largewedge^k V &\simeq  \largewedge^{n-k} V^*  \otimes \largewedge^n V. 
\end{align*}
Moreover, if $0\to U\to V\to W\to 0$ is exact, then
$$ \Dens(U)\otimes \Dens(W)\simeq \Dens(V)$$ 
and $$\ori(U)\otimes \ori(W)\simeq \ori(V).$$

The Euler vector field $E(x)=x$ is defined on any vector space $V$. 

\subsection{Generalized forms}

The theory of distributions on a smooth manifold  is a bit more subtle than that of distributions on open subsets of $\RR^n$. We recall in this section the relevant concepts and refer the reader to \cite[Chapter VI]{GuilleminSternberg:GA} for more information.  

 By $\Dens(M)$ we denote the line bundle over $M$ with fiber $\Dens(T_xM)$ over $x\in M$. Its smooth sections are naturally identified with the smooth measures on $M$. The orientation bundle $\ori (M)$ of $M$ is defined analogously.  The space of generalized sections of a vector bundle $\mathcal E\to M$ is defined by $C^{-\infty}(M,\mathcal E)= (C_c^\infty(M,\mathcal E^*\otimes \Dens(M)))^*$. The space of generalized differential $k$-forms on a manifold $M$ is therefore 
$$  \Omega^k_{-\infty}(M)=  \Omega_c^{n-k}(M, \ori(M))^*= C_c^\infty( M, \largewedge^{n-k}T^*M \otimes  \ori(M))^*.$$
 It contains the dense subspace of smooth forms, given by
\[\langle \omega, \psi\rangle=\int _M \omega\wedge \psi, \quad \omega\in\Omega^k(M), \psi\in  \Omega_c^{n-k}(M, \ori(M)). \]

Most operations on smooth forms extend to generalized forms, including the pullback under a smooth  proper submersion $f\colon M\to N$, the exterior derivative $d$, and  the contraction with a vector field $i_X$. For $\omega\in\Omega^k_{-\infty}(M)$, test forms $\phi\in\Omega^{n-k}_c(M, \ori(M))$, $\eta\in\Omega^{n-k+1}_c(M, \ori(M))$, $\psi\in\Omega^{n-k-1}_c(M, \ori(M))$, and $X$ a vector field we have
\begin{align*}\langle f^*\omega, \phi\rangle&=\langle \omega, f_*\phi\rangle,\\
	\langle i_X\omega, \eta\rangle&=(-1)^{k+1}\langle \omega, i_X\eta\rangle,
	\\\langle d\omega, \psi\rangle&=(-1)^{k+1}\langle \omega, d_\nabla\psi\rangle, 
\end{align*}
where $f_*$ denotes integration along the fibers and  $\nabla$ is the canonical connection on $\ori(M)$.

Let $V, F$ be finite-dimensional real vector spaces.
By a tempered generalized differential $k$-form on $V$ with values in $F$ we understand a continuous linear functional on 
$ \Omega_{\mathcal S}^{n-k}(V,  F^*\otimes \ori(V))$. We denote it by $\Omega_{\mathcal S'}^k(V, F)=\mathcal S'(V, \wedge^k V^*\otimes F)$. 
Equivalently, a tempered generalized form is a generalized form with tempered distributions as coefficients with respect to a basis. Thus
$$\Omega_{\mathcal S'}^k(V, F)\subset \Omega^k_{-\infty}(V,F).$$
In the following we will sometimes have to replace the vector space $F$ by a canonically isomorphic vector space $G$. Let us be clear that this means that if $\omega\in\Omega^k_{\calS'}(V,F)$ is a generalized form, $f\colon F\to G$ is a linear isomorphism, and $\eta\in \Omega^{n-k}_\calS(V, G^*\otimes \ori(V))$ is a test form, then  we define 
\begin{equation}
	\label{eq:changingF}\langle f\circ \omega,\eta\rangle = \langle \omega, f^{\vee}\circ \eta\rangle, 
\end{equation} where $f^\vee \colon G^* \to F^*$ denotes the dual map.

A generalized form $\omega\in\Omega_{-\infty}(V, F)$ is $r$-homogeneous if $m_\lambda^*\omega=\lambda^r\omega$, for all $\lambda>0$. For any $r$, an $r$-homogeneous generalized form is tempered. 

We denote by $\PP_+(V)$ the oriented projectivization of $V$. The manifold $\PP_V= V\times \PP_+(V^*)$ carries a contact structure defined by the hyperplanes $H_{p,[\xi]} = \Ker \xi\circ d\pi$ in  $T_{p,[\xi]} \PP_V$, where $\pi\colon  \PP_V\to V$ denotes the projection to the first factor. A smooth differential form $\omega$ on $\PP_V$ is said to be vertical, if $\omega|_{H_{p,[\xi]}}=0$ for all $p,[\xi]$. A generalized form $\omega$ is called vertical if 
$ \langle \omega, \psi\rangle = 0$ for all vertical test forms $\psi$.

By $\delta_p$ we denote the delta measure at $p$. The corresponding generalized top form, which depends on a choice of orientation $\sigma_p$ at $p$, is $\delta_p\otimes\sigma_p$. Often we will simply write $\delta_p$ also for the top form when the choice of $\sigma_p$ is clear or not important.

\section{$0$-homogeneous forms}
\subsubsection{Homogeneous extension and restriction}

Let us recall how to extend (generalized) forms on $\PP_+(V)$ to $0$-homogeneous tempered generalized forms on $V$. We moreover discuss a  characterization of such extensions. 

Let $\pi\colon V\setminus\{0\}\to \PP_+(V)$ denote the canonical projection. Observe that there is a well-defined push-forward (fiber integration) of 
twisted differential forms
$$ \pi_* \colon \Omega_{\mathcal S}^k(V,  \ori(V)\otimes F) \to \Omega^{k-1}(\PP_+(V), \ori(\PP_+(V))\otimes F).$$
Explicitly, assuming $F=\RR$,  if $\omega= f(r,\theta)  dr\wedge \pi^*\xi$ is a $k$-form, where $\xi\in\Omega^{k-1}(\PP_+(V))$, then $\pi_* \omega= (-1)^{n-k}(\int_0^\infty fdr) \xi$, so that $\langle \pi^* \phi, \omega\rangle=\langle \phi, \pi_*\omega\rangle$ for all  $\phi\in \Omega^{n-k}(\PP_+(V))$.

\begin{Definition}
	Let  $T\in \Omega^{k}_{-\infty}(\PP_+(V),F)$. We define the $0$-homogeneous extension of $T$ to $V$, denoted $r^0T\in\Omega^{k}_{\mathcal S'}(V, F)$ by
	$$ 	\langle r^0T, \omega\rangle = \langle T,\pi_* \omega\rangle, \quad \omega\in \Omega_{ \mathcal S}^{n-k}(V,   \ori(V)\otimes F^*).$$
\end{Definition}
 Clearly the restriction of $r^0T$ to $V\setminus\{0\}$ is $\pi^*T$.
\begin{Remark}
	The notation $r^0T$ is inspired by  Gelfand--Shilov \cite[Section 3.5]{GelfandShilov:I}.
\end{Remark}

\begin{Lemma}\label{lem:supp0}
	Let $S,T\in \Omega^k_{-\infty}(V)$ be $0$-homogeneous generalized forms of degree $k<n$. If $S-T$ is supported at the origin, then $S=T$.
\end{Lemma}
\begin{proof}
	It suffices to prove that $T=0$ if $T$ is supported at the origin. Fixing any basis of $V$, the corresponding coefficients of $T$ are $(-k)$-homogeneous generalized functions that are supported at the origin. Any such function is a linear combination of the delta function and its derivatives, and consequently cannot be $(-k)$ homogeneous for $k<n$, unless it is zero.
Hence $T=0$ as claimed. 
\end{proof}

Recall that $E$ denotes the Euler vector field $E(x)=x$.

\begin{Proposition}\label{lem:zero_homogeneous}
	Let $T\in \Omega^{k}_{-\infty}(\PP_+(V))$. Then $T_0= r^0 T\in\Omega^k_{-\infty}(V)$ has the following properties
	\begin{enumerate}
		\item $i_E T_0 =0$.
		\item $T_0$ is $0$-homogeneous.
	\end{enumerate}
	Conversely, if $T_0\in \Omega^k_{-\infty}(V)$ satisfies the above properties, then either $0\leq k\leq n-1$, whence there exists a unique $T$ such that $r^0 T= T_0$, or $k=n$ and $T$ is a multiple of a delta $n$-form supported at the origin.
\end{Proposition}
\begin{proof}
		It is clear that $T_0=r^0 T$ has the second stated property, and it holds that $i_E(T_0|_{V\setminus\{0\}})=0$.
		Thus $i_ET_0$ is a $0$-homogeneous $(k-1)$-form supported at the origin, but this is impossible for $k-1<n$. Therefore, $i_ET_0=0$.
	
	For the converse statement, fix a Euclidean structure on $V=\R^{n}$, and assume first $k\leq n-1$. The uniqueness of $T$ is straightforward. 
	To prove existence let us first observe that (1) and (2) imply
	$$ i_E d T_0 = \mathcal L_E T_0 = 0.$$
	
	We define $\langle T, \omega \rangle := \langle T_0 ,\phi\rangle $ for $\omega= \pi_* \phi$, where $\phi\in\Omega^{n-k}_c(\R^n\setminus\{0\})$ is a compactly supported test form. For this to be well-defined it suffices to show that $\langle T_0, \phi\rangle =0$ if $\pi_* \phi =0$. We may introduce polar coordinates $r>0,u\in S^{n-1}$ on $V\setminus\{0\}$ and write 
	$$ \phi = f(r,u) dr\wedge  \xi  + \eta,$$
	where $\xi,\eta\in u^*(\Omega(S^{n-1}))$.
	As a consequence,  $\pi_* \phi = (-1)^{k}\left( \int_0^\infty f dr \right) \xi$. 
	By assumption we have $\int_0^\infty f dr=0$, so that $f= \frac{\partial h}{\partial r}$ for the compactly supported function $h(r,u)=\int_0^r f(t,u) dt$. 
	Put $\zeta= h \xi$. 
	Observe that whenever $i_E\psi=0$ and $\theta\in u^*\Omega(S^{n-1})$, and $\psi\wedge \theta$ is a top degree form, one has the pointwise equality $\psi\wedge \theta=0$. Using that $i_ET_0=0$ and $i_E dT_0=0$, we find
	$$ \langle T_0 , \phi\rangle =  \langle T_0,  \frac{\partial h}{\partial r} dr\wedge \xi\rangle = \langle T_0, d\zeta\rangle =(-1)^{k+1} \langle dT_0, \zeta\rangle =0$$ 
	as required. 
	By construction,  $r^0T -T_0$ is supported at the origin.  Applying  Lemma~\ref{lem:supp0} yields $r^0T= T_0$ as desired.
	
	Now if $k=n$, we may write $T_0=f(x)dx_1\wedge \dots\wedge dx_n$ for some $f\in C^{-\infty}(\R^n)$. The condition $i_ET_0=0$ readily implies that $f$ is supported at the origin. As it is also $(-n)$-homogeneous, it must be a multiple of the delta function, completing the proof.
	
\end{proof}
Note that if $T\in\Omega^k(\mathbb P_+(V))$ is a smooth form, then $r^0T$ is smooth on $V\setminus \{0\}$.

\begin{Lemma}\label{lem:tau_L1}
	Let $\tau\in\Omega^k(\mathbb P_+(V))$ and $\omega=r^0\tau\in\Omega^k_{-\infty}(V)$. Then $\omega$ is locally integrable.
\end{Lemma}
\begin{proof}
	Fix a Euclidean structure on $\R^n$.  On $\RR^n\setminus\{0\}$ we may write $\omega=\omega'$, where $\omega'=\frac{1}{|x|^{k}}\sum_{|I|=k} g_I(\frac{x}{|x|}) dx_I$ with $g_I\in C^\infty (S^{n-1})$.  Since $k<n$, the coefficients of $\omega'$  are locally integrable. Applying Lemma~\ref{lem:supp0} yields $\omega=\omega'$ on $\RR^n$.
\end{proof}

\subsection{Operations on $0$-homogeneous forms}

We need to know how to pullback and pushforward $0$-homogeneous generalized forms under linear maps.

\subsubsection{Pullback under linear monomorphisms}

For a vector $v\in\R^n$, we denote by $S_{ v }:\R^n\to\R^n$ the shift $x\mapsto x+ v$.

\begin{Lemma}\label{lem:restrict_not_delta}
	Assume $\omega\in\Omega^{k}_{-\infty}(\R^n)$ is smooth outside of the origin, $0$-homogeneous. Let $e:\R^j\hookrightarrow \R^n$ be a fixed monomorphism, and assume $j>k$. Then the weak limit \[\omega':=\lim_{V\setminus e(\RR^j)\ni v\to  0}e^* S_{ v}^*\omega\] exists, is $0$-homogeneous, and on $\R^j\setminus\{0\}$ coincides with $e^*\omega$. If moreover $i_E\omega=0$ then also $i_E\omega'=0$. If $\omega=r^0\tau$ for $\tau\in\Omega^k(\mathbb P_+(\R^n))$, then one has $\omega'=r^0(e^*\tau)$.
\end{Lemma}
We will subsequently denote $\omega'=e^*\omega$.
\proof
By the smoothness outside of the origin, $\omega'$ coincides with $e^*\omega$ on $\R^j\setminus\{0\}$.
Choose orthonormal coordinates such that $\R^j=\Span(e_1,\dots, e_j)$. 
Observe that since $\omega$ is $0$-homogeneous, Lemma~\ref{lem:supp0} implies \[\omega=\frac{1}{|x|^{k}}\sum_{|I|=k} g_I(\frac{x}{|x|}) dx_I\] for some $g_I\in C^\infty (S^{n-1})$. Then 
\begin{align*}\omega'&= \lim_{v\to 0} \frac{1}{|y+ v|^{k}}\sum_{I\subset\{1,\dots,j\}} g_I(\frac{y+v}{|y+v|}) \wedge_{j\in I} dy_j
	\\&=\frac{1}{|y|^{k}}\sum_{I\subset\{1,\dots,j\}} g_I(\frac{y}{|y|}) \wedge_{j\in I} dy_j.\end{align*}
Thus $\omega'$ is locally integrable and $0$-homogeneous. 

Assume now $i_E\omega=0$. By Proposition~\ref{lem:zero_homogeneous} we may write $\omega=r^0\tau$ for some $\tau\in \Omega^k(\mathbb P_+(\R^n))$. Define $\omega_1=r^0(e^*\tau)$.  Then $\omega_1$ coincides with $\pi^*(e^*\tau)= e^*(\pi^*\tau)=e^*\omega =\omega'$ outside of the origin. Thus $\omega_1=\omega'$ by Lemma~\ref{lem:supp0}  and therefore in particular $i_E\omega'=0$. This concludes the proof.

\endproof

\begin{Proposition}\label{prop:restrict_delta}
Assume $\omega\in\Omega^{k}_{-\infty}(\R^n)$ is smooth outside of the origin, $0$-homogeneous, $k<n$, and $i_E\omega=0$. Let $e:\R^k\hookrightarrow \R^n$ be a fixed monomorphism. Then for $v\in \RR^n\setminus e(\RR^k)$, the weak limit \[e^*_v\omega:=\lim_{\epsilon\to 0^+}e^* S_{\epsilon v}^*\omega\] exists, is uniform in $v\in S^{n-1}\setminus e(\RR^k)$, and equals $c([v])\delta_0 $ for some continuous function $c$ on $\mathbb P_+(\RR^n/e(\RR^k))$. 
Furthermore, if $\omega$ is closed then $c([v])$ is constant.
\end{Proposition} 
 In the latter situation, we denote the limit by $e^*\omega$.
\proof

Write $V=\R^n$, $W=\R^k$.   Without loss of generality we may assume that $e$ is the inclusion of a subspace $W$ of $V$. We choose an orthonormal basis of $V$ such that $W=\Span(e_1,\dots, e_{k})$.

 By Lemma~\ref{lem:supp0} we can find functions $g_J\in C^\infty(S^{n-1})$ such that
\[ \omega = \frac{1}{|x|^{k}}\sum_{|J|=k} g_J(\frac{x}{|x|}) dx_J.\]
Consider first a unit vector $u\perp W$. Denoting $I=\{1,\dots, k\}$, it holds that 
\[e^*(S_{\epsilon u}^*\omega)|_{y} = \frac{1}{|y+ \epsilon u|^{k}} g_I\left(\frac{y+\epsilon u}{|y+ \epsilon u|}\right)dy_1\wedge\dots\wedge dy_{k}.\]
Since $i_E \omega=0$, it holds that $g_I(y)=0$ for 
$y\in W$. Hence 
\begin{equation}\label{eq:taylor} g_I(y+z)= \sum_{i=k+1}^n z_i h_i(y,z), \quad y\in W, \ z\in W^\perp,\end{equation}
where 
$$ h_i (y,z)= \int_0^1 \frac{\partial g_I}{\partial x_i}(y +tz) dt.$$

Denoting  $W_u=W\oplus \Span(u)$, $g_u=g_I|_{S(W_u)}$, we may write $g_u(\theta)=\langle \theta, u\rangle h_u(\theta)$ for some $h_u\in C^\infty(S(W_u))$, which also depends smoothly on $u$. Observe that the family $\{|h_u|: u\in S(W^\perp)\}$ can be uniformly bounded by some constant $B$.

Now let $\psi\otimes \sigma \in C^\infty_c(W)\otimes \ori(W)$ be supported in a ball of radius $R$. Assume for simplicity that the coordinates $y_1,\dots,y_{k}$ are positively oriented with respect to $\sigma$, and we omit $\sigma$ henceforth.

We may write
\[ \langle e^*(S_{\epsilon u}^*\omega), \psi\rangle = \int_{|y|\leq R }\psi(y)\frac{\epsilon}{|y+\epsilon u|^{k+1}} h_u\left(\frac{y+\epsilon u}{|y+\epsilon u|}\right)dy=I_1+I_2,\]
where 
\begin{align*} I_1&=\epsilon\int_{|y|\leq R }(\psi(y)-\psi(0))\frac{1}{|y+\epsilon u|^{k+1}} h_u\left(\frac{y+\epsilon u}{|y+\epsilon u|}\right)dy,\\
I_2&=\psi(0)\epsilon\int_{|y|\leq R }\frac{1}{|y+\epsilon u|^{k+1}} h_u\left(\frac{y+\epsilon u}{|y+\epsilon u|}\right)dy.
\end{align*}
Observe that $|\psi(y)-\psi(0)|\leq C_0|y|$ for some constant $C_0$. Put $M=C_0B$.
We do a change of variables, $y=\epsilon w$, to find that

\begin{align*} |I_1|&\leq \epsilon M\int _{|w|\leq R/\epsilon} \frac{\epsilon |w|}{\epsilon^{k+1}|w+u|^{k+1}}\epsilon^{k}dw=\epsilon M\int_{|w|\leq R/\epsilon}\frac{|w|}{(1+|w|^2)^{(k+1)/2}}dw\\
&\leq C_1\epsilon+\epsilon M\int_{1\leq |w|\leq R/\epsilon}\frac{|w|}{(1+|w|^2)^{(k+1)/2}}dw\leq C_1\epsilon+C_2\epsilon 
\int_{1}^{R/\epsilon}\frac{r^{k}}{(1+r^2)^{(k+1)/2}}dr\\
&\leq C_1\epsilon+C_2\epsilon  \int_{1}^{R/\epsilon}\frac{1}{r}dr= C_1\epsilon+C_2\epsilon \log \frac{R}{\epsilon}.
\end{align*}
Thus $I_1\to 0$ as $\epsilon\to0^+$, uniformly in $u\in S(W^\perp)$.

Next we do the same change of variable $y=\epsilon w$ for $I_2$. We find that 

\begin{align*} I_2&=\psi(0)\epsilon\int_{|y|\leq R }\frac{1}{|y+\epsilon u|^{k+1}} h_u\left(\frac{y+\epsilon u}{|y+\epsilon u|}\right)dy\\
&=\psi(0)\int_{|w|\leq R/\epsilon} \frac{1}{(1+|w|^2)^{(k+1)/2}}h_u\left(\frac{w+u}{|w+u|}\right)dw.\end{align*}
The last integral converges as $\epsilon \to 0$, uniformly in $u\in S(W^\perp)$, to the absolutely convergent integral
\begin{equation}\label{eq:delta_constant}c(u):=\int_{\R^{k}} \frac{1}{(1+|w|^2)^{(k+1)/2}}h_u\left(\frac{w+u}{|w+u|}\right)dw,\end{equation}
and so $I_2\to  c(u)\psi(0)$. We conclude that $e^*(S_{\epsilon  u}^*\omega)\to c(u)\delta_0$ as $\epsilon\to 0^+$, uniformly in $u\in S(W^\perp)$. As  a trivial consequence we obtain that  
$e^*(S_{\epsilon  \lambda u}^*\omega)\to c(u)\delta_0$ uniformly in  $\{(u,\lambda): u\in S(W^\perp), 0< \lambda\leq 1\}$.

Now for a general $v\in S(V)\setminus W$, there are unique $w\in W$, $0<\lambda\leq 1$ and $u\in S(W^\perp)$ such that $v=\lambda u+w$.  
We have $S_{\epsilon v}=S_{\epsilon \lambda u}\circ S_{\epsilon w}$, so that 
\[\langle e^*(S_{\epsilon v}^*\omega), \psi\rangle =\langle e^*(S_{\epsilon w}^*\circ S_{\epsilon \lambda u}^* \omega), \psi\rangle=\langle e^* ( S_{\epsilon \lambda u}^*\omega), S_{-\epsilon w}^*\psi\rangle.\]
Tracing the proof above with $\psi$ replaced by $S_{-\epsilon w}^*\psi$, we see that $|I_1|\to 0$ uniformly in $v$, while
\[ I_2=\psi(-\epsilon w)\int_{|w|\leq R/(\lambda\epsilon)} \frac{1}{(1+|w|^2)^{(k+1)/2}}h_u\left(\frac{w+u}{|w+u|}\right)dw \]
converges to $c(u)\psi(0)$ uniformly in $v$.  In particular, the limit only depends on $[v]\in \mathbb P_+(V/W)$. The continuity of $c(u)$ on $S(W^\perp)$ is seen from  formula \eqref{eq:delta_constant}.

Finally, assume that $\omega$ is closed. 
By Proposition~\ref{lem:zero_homogeneous} we have $\omega=r^0\tau$ for some $\tau\in\Omega^{k}(S^{n-1})$. Since $0=d\omega = d\pi^* \tau = \pi^* d \tau $ on $V\setminus\{0\}$, where $\pi\colon V\setminus\{0\}\to S^{n-1}$ denotes the radial projection, we conclude that $\tau$ is closed.   Fix $\psi\in C^\infty_c(W)\otimes \ori(W)$.

For every $v\in V\setminus W$ put $W_v= W\oplus \operatorname{span}(v)$. For every $\epsilon >0$, the image of 
$\epsilon v + W$ under $\pi$ is the open hemisphere $S(W_v)$ containing $\pi(v)$, with boundary $S(W)$. Let $S_v^+(W_v)$ denote the corresponding closed hemisphere. Put $\pi_{\epsilon,v}=\pi|_{\epsilon v + W}$. Let $\psi_{\epsilon,v}\in  C^\infty_c(S(W_v), \ori(S(W_v))$ be such that 
$ (\pi_{\epsilon,v})^* \psi_{\epsilon, v} = S^*_{-\epsilon v}\psi$.  
It then holds that 
$$ \langle e^* S^*_{\epsilon v} \omega, \psi \rangle = \int_{\epsilon v+W} \pi^*\tau  \otimes  \pi^*\psi_{\epsilon,v} = \int_{S_v^+(W_v)} \tau \otimes \psi_{\epsilon,v}.
$$

Observe that 
$$ \psi_{\epsilon, v} \to \psi(0) \cdot \mathbf{1}_{S_v^+(W_v)} $$  pointwise as $\epsilon$ goes to zero. Hence by the dominated convergence theorem we obtain 
\begin{equation}\label{eq:const}c(v)= \lim_{\epsilon\to 0}\int_{S(W_v)} \tau \otimes \psi_{\epsilon,v} \to \int_{S_v^+(W_v)} \tau \otimes \sigma_v \otimes  \psi(0),
\end{equation}
 where $\sigma_v(v)=1$. 
 Since $d\tau =0$,  the latter integral remains constant if $v$ follows a continuous path in $S^{n-1}\setminus W$. By the same token,  $\int_{S_v^+(W)} \tau + \int_{S_{-v}^+(W)} \tau =\int_{S(W_v)} \tau$. Since $\sigma_{-v}= -\sigma_v$, it follows that $c(-v)=c(v)$. This concludes the proof.

\endproof

\endproof

\begin{Lemma}\label{lem:pullback_forms}
Assume $\omega\in\Omega^{k}_{-\infty}(\R^n)$ is a closed generalized form, smooth outside of the origin, $0$-homogeneous,  and $i_E\omega=0$. Let $e:\R^j\hookrightarrow \R^n$ be a fixed monomorphism. Let $\mu_\epsilon(x) dx\in\mathcal M^\infty(\R^n)$ be an approximate identity, with $\mu_\epsilon$ a Schwartz function. Then the weak limit \[\lim_{\epsilon\to 0^+}e^* (\omega\ast\mu_\epsilon)\] exists and equals $e^*\omega$.
\end{Lemma}
\proof
For $k<j$, this follows from Lemma \ref{lem:restrict_not_delta}, while for $k=j$ we use Proposition \ref{prop:restrict_delta}.
Let us provide the details in the case $k=j$, which is more involved. 

By Proposition  \ref{prop:restrict_delta}, the limit $e^*\omega=\lim_{\R^n\setminus \R^j\ni v\to 0} e^*S^*_v\omega$ exists.  
Let $\psi\in \Omega_c(\R^k)$ be a test form. It then holds that
\[ f(v):=\langle e^*S_v^*\omega, \psi\rangle-\langle e^*\omega, \psi\rangle:\R^n\setminus \R^j\to \R \]
satisfies $f(v)\to 0$ as $|v|\to 0$. 

Furthermore, $f$ is bounded. Indeed, write $v=\lambda u+w$, where $w\in W$ and $u\in S(W^\perp)$. Write
\[\langle e^*S_v^*\omega, \psi\rangle=\langle e^*S_{\lambda u}^*\omega, S^*_{-w}\psi\rangle=I_1+I_2  \] as in the proof of Proposition \ref{prop:restrict_delta}. As $|S^*_{-w}\psi(y)-S^*_{-w}\psi(0)|\leq C_0|y|$ for some $C_0$ that is independent of $w$, while $|\psi(-w)|$ is uniformly bounded, it follows that $I_1, I_2$ and therefore $\langle e^*S_v^*\omega, \psi\rangle$ is bounded as $\lambda \to 0$, uniformly in $w$ and $u$. The same estimates for $I_1, I_2$ in the proof of Proposition \ref{prop:restrict_delta} show that $I_1, I_2$ remain bounded uniformly in $u, w$ also for $|\lambda|\geq \lambda_0>0$.

It follows that $\langle e^*S_v^*\omega, \psi\rangle$ is bounded as a function of $v\in \R^n\setminus \R^j$, and therefore so is $f(v)$.

Choose a function $\delta(\epsilon)>0$  such that $\delta(\epsilon)\to 0$ as $\epsilon\to 0$, and $\int_{|v|\geq \delta(\epsilon)}\mu_\epsilon(v)\leq \delta(\epsilon)$. It then holds that

\begin{align*}
	\langle e^*(\omega\ast\mu_\epsilon), \psi\rangle -\langle e^* \omega, \psi\rangle =\int \langle e^*S_v^* \omega-e^*\omega, \psi\rangle d\mu_\epsilon(v)=
	\int_{\R^n} f(v)d\mu_\epsilon(v).
	\end{align*}

Seeing that $\left|\int_{|v|\geq \delta(\epsilon)} f(v)d\mu_\epsilon(v)\right|\leq \sup_{\R^n} |f(v)| \delta(\epsilon)$, while 
\[\int_{|v|\leq \delta(\epsilon)} f(v)d\mu_\epsilon(v)\leq \sup_{|v|\leq \delta(\epsilon)}|f(v)| \to 0, \]
it follows that 
\[	\langle e^*(\omega\ast\mu_\epsilon), \psi\rangle \to\langle e^* \omega, \psi\rangle,\]
as claimed.
\endproof

\subsubsection{Pushforward under linear epimorphisms}

 In general, generalized forms do not pushforward under  linear epimorphisms. 
The following definition is motivated by the description of the pullback of valuations in terms of differential forms given in \cite{alesker_radon}. 

\begin{Definition}\label{def:pushforward}
Let $\omega\in\Omega^{ k}_{-\infty}(V)$ be smooth outside of the origin, $0$-homogeneous, and satisfy $i_E\omega=0$, and let  $p:V\to W$ be  a linear epimorphism.  We define the \emph{pushforward} of $\omega$ as
$$p_*\omega:=r^0(\tilde p_*\tilde \tau) \in \Omega^{k-s}_{-\infty}(W,\ori(\Ker p)),$$
 where  $\tilde p:\widetilde {\mathbb P_+(V)}\to \mathbb P_+( W)$ is the induced projection on the oriented blow-up of $\mathbb P_+(V)$ along $\mathbb P_+(\Ker p)$,   $\tau$ satisfies $\omega= r^0\tau$, and  $\tilde \tau=\pi^*\tau$ where $\pi:\widetilde {\mathbb P_+( V)}\to \mathbb P_+(V)$ is the natural projection.

 Furthermore, for $\omega\in \Omega^{n}_{-\infty}(V)$ equal to $c\delta_0\otimes \sigma_V$, we define $p_*\omega=c\delta_0\otimes \sigma_W\otimes \sigma_p$, where $\sigma_V\in\ori(V), \sigma_W\in\ori(W), \sigma_p\in\ori(\Ker p)$ are such that $\sigma_V=\sigma_W\otimes\sigma_p$. \end{Definition}

We will need the following description of the pushforward.

\begin{Lemma}\label{lem:pushforward_forms}
Let $p: V\to  W$ be a linear epimorphism,  $\dim V=n$, $\dim W=n-s$. Assume that $\omega\in\Omega_{-\infty}^{k}(V)$ is $0$-homogeneous, smooth outside of the origin, and $i_E\omega=0$. Let $\nu_\epsilon \in C^\infty( V )$ satisfy $\nu_\epsilon \to 1$ smoothly on compact sets, with $\nu_\epsilon$ a uniformly bounded family of Schwartz functions. Then $p_*\omega \in\Omega_{-\infty}^{k-s}(W , \ori(\Ker p))$ is given by 
\[\langle \psi, p_*\omega\rangle =\lim_{\epsilon\to 0} \langle \nu_\epsilon\cdot p^*\psi, \omega \rangle ,\] for all compactly supported $\psi\in \Omega^{ n-k}(W  , \ori(V))$.
\end{Lemma}
\proof
For $k=n$ we have $\omega=c\delta_0$, and both sides trivially equal $c\psi(0)$. 

Assume now $k<n$ and  that $V=\RR^n$, $W=\RR^{n-s}$ and  coordinates are chosen so that $\Ker p = \{x_1=\dots=x_{ n-s}=0\}$. Write \[\omega=\frac{1}{|x|^k}\sum_{|I|=k} g_I(\frac{x}{|x|})dx_I=:\sum \omega_I	.\]
We first show that \[ \int_{\R^n} p^*\psi\wedge\omega\]
is absolutely convergent, and so by the dominated convergence theorem, the limit on the right hand side of the claimed equality exists and equals this integral. 

When $k>s$, this follows at once since $\frac{1}{|x|^k}$ is integrable over $\R^s$.

Now assume $k=s$. Denoting $I=\{n-k+1,\dots, n\}$, we may assume $\omega=\omega_I$. Examining the coefficient of $dx_{n-k+2}\wedge\dots\wedge dx_n$ in the equality $i_E\omega=0$, and writing $I_j=\{j, n-k+2,\dots, n\}$ for $1\leq j\leq n-k$, we find that
\[ g_I(\frac{x}{|x|})x_{n-k+1}=-\sum_{j=1}^{n-k}g_{I_j}(\frac{x}{|x|})x_j.\]
Denoting $x=y+z$ where $y\in\R^{n-k}$ and $z\in\R^k$, we write this equality as
\[  g_I(\frac{x}{|x|})=\frac{f_{1,y}(z)}{x_{n-k+1}}=\frac{f_{1,y}(z)}{z_{1}} \]
for some $f_{1,y}(z)\in C^\infty(\R^k\setminus 0)$, which is bounded for each fixed $y$, and smooth in $\R^k$ when $y\neq 0$. We similarly can find $f_{i,y}(z)\in C^\infty(\R^k\setminus 0)$ for $1\leq i\leq k$ such that 
\[ g_I(\frac{x}{|x|})=\frac{f_{i,y}(z)}{z_{i}}.\]
Note also that $|f_{i, y}(z)|\leq c_0|y|$ for some fixed $c_0>0$.

We have \begin{align*} p^*\psi \wedge\omega& =\frac{1}{|x|^{k}} g_I(\frac{x}{|x|}) p^*\psi\wedge dx_{n-k+1}\wedge\dots\wedge dx_n 
\\&=h(y,z) dy_1\wedge\dots\wedge dy_{n-k}dz_1\wedge\dots\wedge dz_k\end{align*}
for some $h(y,z)$ smooth outside $(0,0)$ and supported in $|y|<R$ for some $R>0$.

Thus for all $(y,z)\in\R^{n-k}\oplus \R^k$ with $y,z\neq 0$ it holds that

\[ |h(y,z)|\leq  \frac{c_0|y|}{|y+z|^k \max(|z_1|,\dots, |z_k|)}\leq \frac{c_1|y|}{|y+z|^k|z|}. \]

Consequently,
\[ \int_{\R^k}|h(y,z)|dz\leq c_1|y|\int_{\R^k}\frac{dz}{|z||y+z|^k}=c_2|y|\int_0^\infty\frac{r^{k-2}dr}{(r^2+|y|^2)^{k/2}}=c_3\]
and so $\int_{\R^n}|h(y,z)|dydz\leq \int_{|y|\leq R} c_3 dy$ is finite.

Thus in all cases it remains to verify that
\[\langle \psi, p_*\omega\rangle = \langle p^*\psi, \omega\rangle. \]

Let $\omega=r^0\tau$. 
To avoid dealing with the blow-up, we may assume that $\tau$ vanishes near $\Ker p$, and conclude using approximation. Denoting by $\pi:V\setminus\{0\}\to \mathbb P_+(V)$ the radial projection, the left hand side is by definition
\[\langle \psi, p_*\omega\rangle=\langle \psi,  r^0p_*\tau\rangle=\langle \pi_*\psi,  p_*\tau\rangle=\langle p^*\pi_*\psi, \tau \rangle.\]
On the other hand, since $\tau$ vanishes near $\Ker p$, \[\langle p^*\psi,\omega\rangle=\langle \pi_*p^*\psi, \tau\rangle.\]
It remains to observe that  
 the top arrow in 
	\begin{center}
	\begin{tikzpicture}
		\matrix (m) [matrix of math nodes,row sep=3em,column sep=4em,minimum width=2em]
		{
			V\setminus \Ker p  & \PP_+(V\setminus \Ker p) \\
			W & \PP_+(W) \\};
		\path[-stealth]
		(m-1-1) edge node [left] {$p$}    (m-2-1)
		edge  node [above] {$\pi$} (m-1-2)
		(m-1-2) edge node [right] {$p$} (m-2-2)
		(m-2-1) edge node [above] {$\pi$} (m-2-2);
	\end{tikzpicture}
	\end{center}
is a bundle map that restricts to a diffeomorphism on each fiber, and so $\pi_*p^*\psi=p^*\pi_*\psi$. 
\endproof

\section{The $0$-homogeneous current of a valuation}

A key property of smooth valuations is their description through a pair of currents  $(C,T)$, obtained in \cite{AleskerBernig:Product}, equations (60)--(62). Here we unify them into a single current in the translation-invariant case. This simplifies the description of some operations on valuations such as pullback and pushforward, and is very natural for describing the  Alesker--Fourier transform on valuations. 

 In the following we make frequent use of the Hodge star isomorphism 
\begin{equation}\label{eq:ident}
	\ast=\ast_V: \largewedge^k V^* \xrightarrow{\sim} \largewedge^{n-k} V \otimes \largewedge^n V^* \simeq  \largewedge^{n-k} V \otimes \Dens(V) \otimes \ori(V),\end{equation}
given by $\langle \eta, \zeta\rangle=\eta\wedge \ast\zeta,\quad \zeta\in\wedge^k V^*,\eta\in\wedge^k V$.
In view of later computations,  we remark that the Hodge star isomorphism satisfies 
	$$ \zeta \wedge \theta= \langle \ast \zeta, \theta\rangle ,\quad \zeta\in \largewedge^k V^*, \theta\in \largewedge^{n-k} V^*$$
	and 
	$$ (\ast_{V^*} \otimes \id) \circ \ast_V =(-1)^{k(n-k)}\id.$$

\subsection{The $0$-homogeneous current of a valuation}\label{sec:valuation_current}
 Recall that by definition a generalized valuation $\phi\in \Val_k^{-\infty}(V)$ is a continuous linear functional on $\Val_{n-k}^\infty(V)\otimes \dens(V)^*$.  Hence $\phi$ defines for $k>0$ a linear functional on $\Omega^{k-1}(\PP_+(V^*), \largewedge^{n-k} V^* \otimes \ori(V) \otimes \dens(V)^*)$, that is a generalized form in $S\in \Omega_{-\infty}^{n-k}(\PP_+(V^*), \largewedge^{n-k} V \otimes \dens(V))$. The Hodge star isomorphism \eqref{eq:ident} allows to consider it as a  generalized form $$T=T(\phi)=(-1)^{k(n-k)}  *_{V^*}\circ S\in  \Omega_{-\infty}^{n-k}(\PP_+(V^*), \largewedge^{k} V^* \otimes \ori(V)).$$ This is the current $T$ of a valuation mentioned above, specialized to the translation-invariant case.
 
 	\begin{Remark}
 		 Note that our convention differs slightly from that of the standard one, see e.g.\ \cite{AleskerBernig:Product}. Namely, let $\phi\in\Val_k^\infty(V)$, $V=\R^n$, be given by $\phi(K)=\int_{\nc(K)}\omega$ for $\omega\in\Omega^{k,n-1-k}(V\times \mathbb P_+(V^*), \ori(V))$. Then \[T(\phi)\in \Omega^{n-k}(\PP_+(V^*), \largewedge^{k} V^* \otimes \ori(V))=\Omega^{k,n-k}(V\times \mathbb P_+(V^*), \ori(V))\] is given by
 		 \begin{equation}\label{eq:extra_sign}T(\phi)=(-1)^{n-k}a^*D\omega,\end{equation}
 		 where $a$ is the antipodal map, and $D$ the Rumin differential.
 		 The extra sign $(-1)^{n-k}=(-1)^{(n-k)^2}$ is due to the sign difference between the natural pairing
  		 \begin{align*} \Omega^{n-k}(\PP_+(V^*), \largewedge^{k} V^* \otimes \ori(V))&\otimes \Omega^{k-1}(\PP_+(V^*), \largewedge^{n-k} V^* \otimes \ori(V)\otimes\Dens(V)^*)\to \CC,\\ 
  		 (\alpha\otimes dx_I)&\otimes (\beta\otimes dx_J)\mapsto \left(\int_{\mathbb P_+(V^*)} \alpha\wedge \beta\right)\otimes (dx_I\wedge dx_J),\end{align*}
  			 and the single wedge product pairing
 		 \[ \Omega^{k,n-k}(V\times \PP_+(V^*), \ori(V))\otimes \Omega^{n-k,k-1}(V\times\PP_+(V^*), \ori(V)\otimes\Dens(V)^*)\to \CC.\]
 	\end{Remark}

For a generalized translation-invariant valuation $\phi\in\Val^{-\infty}_k(V)$, we define its \emph{$0$-homogeneous current} $\tau(\phi)\in\Omega_{-\infty}^{n-k}(V^*, \wedge^k V^*\otimes \ori(V))$ as follows. If $k>0$, then $\tau(\phi)=r^0T$.  If $k=0$, then $\phi=c\chi$ and we set $\tau(\phi)=c\delta_0$, where $\delta_0$ is the delta measure at the origin. 

\begin{Proposition}\label{prop:forms_121_valuations}
	A generalized form $\tau\in\Omega_{-\infty}^{n-k}(V^*, \wedge^k V^*\otimes \ori(V))$ is $\tau(\phi)$ for some $\phi\in\Val^{-\infty}_k(V^*)$ if and only if $\tau$ satisfies
	\begin{enumerate}\label{valuation_type}
		\item $0$-homogeneous.
		\item $i_E\tau=0$.
		\item $\tau\wedge E=0$.
		\item $d\tau=0$.
	\end{enumerate}
	Moreover, $\phi$ is smooth if and only if $\tau(\phi)$ is smooth on $V^*\setminus\{0\}$.
\end{Proposition}
\begin{proof}
	
	Assume $\tau=\tau(\phi)$. For $k=0$ the verification is trivial, thus we assume $k\geq 1$. Write $\tau=r^0T$ as in the definition of $\tau(\phi)$. By Lemma \ref{lem:zero_homogeneous}, $\tau$ satisfies the first two properties. 
	
	It holds on $V^*\setminus\{0\}$ that $d\pi^*T=\pi^*dT=0$. Therefore, the generalized form $d\tau$ is supported at the origin. For a fixed functional $\psi\in (\wedge^k V^*\otimes \ori(V))^*$, $\phi(d\tau)$ a $0$-homogeneous generalized $(n-k+1)$-form, and so it must hold that $n-k+1=n$, and $\phi(d\tau)$ is a multiple of the delta top form at the origin, denoted $\delta_0$. But $\delta_0$ does not vanish on constant functions, while $d \tau$ must vanish on constants. It follows that $d\tau=0$.
	
	Finally, for a test form $\psi\in \Omega^k(V^*, \wedge^{n-k-1}V^*)$ it holds that
	\[ \langle \tau, \psi\wedge E\rangle  =\langle T, \pi_*(\psi\wedge E)\rangle =0\]
	by the verticality of $T$. Therefore, $\tau\wedge E=0$.
	
	Conversely, let $\tau$ satisfy all four properties. Assume first $k=0$. Since $\tau\wedge E=0$, $\tau$ must be supported at the origin. Since it is $0$-homogeneous, $\tau$ is a multiple of the delta measure, and so $\tau=\tau(c\chi)$ for some constant $c$. 
	
	Assume now $k\geq 1$. By Lemma \ref{lem:zero_homogeneous}, $\tau=r^0\omega$ for a unique generalized form $\omega$ on $\mathbb P_+(V^*)$. It follows immediately from $\tau\wedge E=0$ that $\omega$ is vertical. Also on $V^*\setminus\{0\}$ we have $\pi^*d\omega=d\pi^*\omega=d\tau=0$, and therefore $d\omega=0$. The existence of $\phi$ such that $\tau=\tau(\phi)$ now follows from \cite{AleskerBernig:Product}.
	
\end{proof}

\begin{Definition}
	A form satisfying the properties listed in Proposition \ref{prop:forms_121_valuations} will be called  \emph{valuation-type}. If it is smooth outside the origin, we shall call it  \emph{smooth valuation-type}.
\end{Definition}

 The following special case  will be relevant for us later.

\begin{Lemma} \label{lemma:tauProj}
Let $U$ be a $k$-dimensional linear subspace of $V=\R^n$ and let $P\colon V\to V/U$ denote the canonical projection. Let $\vol_{V/U}$ be a density on $V/U$. Then $\vol_{V/U}$ defines a continuous valuation \[\phi(K)= \vol_{V/U}(PK), \quad K\subset V.\]
 It then holds that 
$$ \tau (\phi)=  P^*\vol_{V/U}  \otimes [[U^\perp]] .$$
\end{Lemma} 
We remark that $[[U^\perp]]\in\Omega^k_{-\infty}(V^*, \ori(U^\perp)\otimes \ori(V))=\Omega^k_{-\infty}(V^*, \ori(U))$, 
 while $P^*\vol_{V/U}\in \largewedge^{n-k}(V/U)^*\otimes \ori(V/U)\subset \largewedge^{n-k}V^*\otimes\ori(V)\otimes\ori(U)^*$, so that $P^*\vol_{V/U}  \otimes [[U^\perp]]$ is well-defined.
\begin{proof} 
 Let $V=\RR^n$ with the standard Euclidean structure and  orientation, and suppose that $U=\{x_{k+1}=\ldots= x_{n}=0\}$ and $\phi(K)=\vol_{U^\perp}(PK)$, where $P$ is the orthogonal projection onto $U^\perp$. Let $L\subset U$ be a convex body of unit volume.
	The valuation $\phi\in\Val_{n-k}(V)$ is continuous, but not smooth. Its action on $\psi\in \Val_{k}^\infty(V)$ is
	$ \langle \phi, \psi\rangle = \psi(L)$.
	If $\psi= \int_{\operatorname{nc}(K)} \omega$, then by eq. \eqref{eq:extra_sign} we have
	$$ \langle \phi,\psi\rangle = (-1)^{k} \int_{L\times S(U^\perp)} \omega $$
	and hence
	$$ T(\phi)= (-1)^{k+ k(n-k)} dx_{k+1}\wedge \cdots \wedge dx_{n} \otimes [[ S(U^\perp)]].$$
	We conclude that 
	\begin{equation}\label{eq:tauProj} \tau(\phi)= r^0(T(\phi))= (-1)^{k(n-k)} dx_{k+1}\wedge \cdots \wedge dx_{n} \otimes [[ U^\perp]].\end{equation}

\end{proof}

\subsection{Operations on valuations}

Recall for the following that for any generalized form $\omega\in\Omega_{-\infty}(V)$, $i_*\omega$ and $p^*\omega$ are well-defined generalized forms, where $i:V\to W$ denotes any monomorphism, and $p:W\to V$ any epimorphism. 

Here we describe some operations on valuations in terms of the $0$-homogeneous current representation. Given a linear map $f:V_1\to V_2$, one can define the pullback $f^*:\Val(V_2)\to \Val(V_1)$, and the pushforward $f_*:\Val(V_1)\otimes\Dens(V_1^*)\to \Val(V_2)\otimes\Dens(V_2^*)$, see \cite{Alesker:Fourier}.

Recall that if $f$ is a monomorphism, then $f^*$ maps smooth valuations to smooth valuations, while if $f$ is an epimorphism, then $f^*$ extends by continuity to a linear map 
$$ f^*\colon \Val^{-\infty}_k(W)\to \Val^{-\infty}_k(V).$$

 Let $n_V=\dim V$ and  $n_W=\dim W$ in the following.
\begin{Proposition}\label{prop:pullback_linear}
	Let $f\colon V\to W$ be a linear map. The pullback $f^* \colon \Val^{\infty}_k(W)\to \Val^{-\infty}_k(V)$ is then given by 
	$$ \tau(f^*u )=  f^\vee \circ  (f^\vee)_* \tau(u),  $$
	where  $(f^\vee)_*:\Omega_{-\infty}^{n_W-k}(W^*, \wedge^k W^*\otimes \ori(W))\to \Omega^{n_V-k}_{-\infty}(V^*, \wedge^k W^*\otimes \ori(V))$ is the pushforward by $f^\vee \colon W^*\to V^*$, and $f^\vee:\wedge^k W^*\to \wedge ^k V^*$ is applied to the value of the form.
\end{Proposition}

\begin{proof}
	Assume first $k=0$. We then have $\tau(\chi)=\delta_0=[[\{0\}]]$.  If  $f$ is either an  epimorphism  or a monomorphism, then \[(f^\vee)_*(\delta_0)=\delta_0.\] 
		Indeed, in the first case the pushforward is the usual pushforward under proper immersions. In the second case, our Definition~\ref{def:pushforward} applies. Since  $f^*\chi=\chi$, the conclusion for $k=0$ follows.
	 All other cases follow immediately from \cite{alesker_radon}, adapted to the translation-invariant case  and bearing in mind the additional sign of eq. \eqref{eq:extra_sign} imposed by our convention for the current of a valuation.
\end{proof}

Recall that if $f$ is an epimorphism, then $f_*$ maps smooth valuations to smooth valuations, while if $f$ is a monomorphism, then  $f_*$ extends by continuity to a linear map
$$ f_*\colon \Val^{-\infty}_{k}(V)\otimes \Dens(V^*)\to \Val^{-\infty}_{k+n_W-n_V}(W)\otimes \Dens(W^*).$$

\begin{Proposition}\label{prop:pushforward_linear}
	Let $f\colon V\to W$ be a linear map. 	For $0\leq k\leq n_V$ and $\phi\in  \Val_k^{-\infty}(V)\otimes \Dens(V^*)$, we have 
	\[\tau(\phi)\in\Omega^{n_V-k}_{-\infty}(V^*, \wedge^k V^*\otimes \wedge^{n_V} V)\]
	and so	
	$$ \widetilde \tau(\phi):=\ast_V\circ\tau(\phi)\in \Omega_{-\infty}^{n_V-k}(V^*, \wedge^{n_V-k} V).$$
	The pushforward 
$$f_*\colon \Val^{\infty}_{k}(V)\otimes \Dens(V^*)\to \Val^{-\infty}_{k+n_W-n_V}(W)\otimes \Dens(W^*)$$
	is then given by
		\begin{equation}\label{eq:f_*} \widetilde\tau(f_*\phi )= f \circ (f^\vee)^* \widetilde\tau(\phi)\in \Omega^{n_V-k}_{-\infty}(W^*, \wedge^{n_V-k} W)\end{equation}
	if $k+n_W>n_V$. If $k+n_W=n_V$, then
	\begin{equation} \label{eq:f_*2} \widetilde\tau(f_*\phi )=   f \circ (f^\vee)_v^* \widetilde\tau(\phi)\in \Omega^{n_W}_{-\infty}(W^*, \wedge^{n_W} W),\end{equation}
	where the choice of $v\in V^*\setminus f^\vee(W^*)$ can be arbitrary.
\end{Proposition}

\begin{proof}
	 Suppose $k>0$ and $k+n_W-n_V>0$.
	 If $f$ is an epimorphism, then the statement follows directly from \cite[Lemma 7.4]{KotrbatyWannerer:Harmonic}. If $f$ is a monomorphism, then using the fact that $f_*$ is defined as the dual of $f^*$, the claim follows from a short computation. 
	
	Let us verify the statement for $k=0$, that is for $f_*(\chi)$. Since every linear map factors as $V\to f(V)\to W$, we have that $f_*\chi=0$ if $f$ is not injective. In this case also right-hand side of \eqref{eq:f_*} vanishes.  Hence we may assume that $f$ is a monomorphism. Fix orientations and Lebesgue measures on $V$ and $W$ for simplicity. This also fixes orientations and Lebesgue measures on $f(V)$ and $W/f(V)$ and elements in the top exterior powers of $V$ and $f(V)$ that we denote by $\sigma_V$ and $\sigma_{f(V)}$.  One has $f\otimes (f^\vee)^*(\sigma_V \otimes [[0]] )=f(\sigma_V)\otimes [[\Ker f^\vee]] $
	, while $$f_*\chi (K) =\int_{W/f(V)}\chi((f(V)+w)\cap K)dw= \vol(P_{W/f(V)} K)  $$
and so by Lemma~\ref{lemma:tauProj} one has $\widetilde\tau(f_*\chi)=\sigma_{f(V)} \otimes [[f(V)^\perp]] =f(\sigma_V)\otimes [[\Ker f^\vee]] $, as required.

It remains to verify the statement for $k=n_V-n_W$. Since every linear map factors as $V\to f(V)\to W$, the pushforward $f_*\phi$ vanishes if  $f$ is not surjective. In this case also right-hand side of \eqref{eq:f_*2} vanishes.  We assume therefore that $f$ is an epimorphism.  For convenience fix a Euclidean structure on $V$, and orientations on $V$ and $W$, such that $v$ is orthogonal to $f^\vee(W^*)$. Assume $\widetilde\tau=\widetilde\tau(\phi)=r^0\widetilde\beta$ with $\widetilde\beta\in \Omega^{n_W}(S(V^*), \wedge ^{n_W}V)$. 
Note that  $\widetilde\beta=d\widetilde\omega$ for some $\widetilde\omega\in \Omega^{n_W-1}(S(V^*), \wedge ^{n_W}V)$. Then $f_*\phi=\chi\otimes \sigma$, where $\sigma=f(\int_{S(W^*)}\widetilde\omega)\in  \wedge ^{n_W}W$. Hence 
$$ \widetilde\tau(f_* \phi) = f(\int_{S(f^\vee (W^*))}\widetilde\omega) \delta_{0},$$
and comparing with \eqref{eq:const} completes the proof.

 \end{proof}

\begin{Proposition}\label{prop:exterior_product}
	Let $\phi\in\Val_k^\infty(V)$ and $\psi\in\Val_l^\infty(W)$. Then $\tau(\phi\boxtimes \psi)= (-1)^{(n-k)l}\tau(\phi)\boxtimes\tau(\psi)$.
\end{Proposition}

 Here the factor $(-1)^{(n-k)l}$ accounts for the difference between viewing forms on $\PP_+(V^*)$ with values in $\largewedge^\bullet V^*$ as translation-invariant forms on the cosphere bundle $V\times \PP_+(V^*)$, which is the setting of \cite{AleskerBernig:Product}, when taking wedge products. 

\proof
We may assume that $\phi\in\Val^\infty_k(V)$, $\psi\in\Val^\infty_l(W)$. Fix Euclidean inner products. 
Assume first $k, l>0$, so that $\tau(\phi)$, $\tau(\psi)$ are locally integrable. Let us recall the construction of the exterior product of $\phi\in \Val^\infty(\R^m)$ and  $\psi\in \Val^\infty(\R^n)$ in these terms. On the blow-up space 
$$ \Sigma:= S^{m-1}\times [0,\pi/2] \times S^{n-1}$$
one has the blow-down map 
$$F:\Sigma \to S^{m+n-1},\quad F(u,\theta, v)= (\cos(\theta) u, \sin(\theta) v)\in S^{m+n-1}$$
and the projection 
$$ \Phi:\Sigma\to  S^{m-1}\times S^{n-1},\quad \Phi(u,\theta, v) = (u,v).$$
The defining generalized $n$-form of $\phi\boxtimes \psi$ is 
$$ T(\phi\boxtimes\psi)=F_* \Phi^* ( T(\phi)\boxtimes T(\psi))\in \Omega_{-\infty}(\mathbb P_+(V^*\times W^*), \wedge^\bullet (V^*\times W^*)).$$ 

We use $\pi$ to denote the radial projection in all spaces. 
Define \[F_0\colon (0,\infty)\times \Sigma\to \R^{m+n},\quad F_0(t,u,\theta,v)=  (t \cos(\theta) u, t\sin(\theta) v)\] so that 
$\pi \circ F_0 = F$.  Define also $\beta:(0,\infty)\times\Sigma\to \Sigma$, $\beta(r, \sigma)=\sigma$. 
\[\Phi_0:(0,\infty)\times \Sigma\to S^{m-1}\times S^{n-1}, \quad \Phi_0=\Phi\circ \beta.\]

Finally, define the map
$$  R\colon \R^m\setminus\{0\}  \times \R^n\setminus\{0\}\to (0,\infty)\times \Sigma $$
so that  $F_0\circ R = \id$ and $\Phi_0\circ R = \pi \times \pi$.
Observe also that $F$, $R$, $F_0$ are all diffeomorphisms outside of submanifolds of positive codimension.

As both sides of the claimed equality 
\[ r^0F_* \Phi^* ( T(\phi)\boxtimes T(\psi))=  (-1)^{(n-k)l} r^0T(\phi)\boxtimes r^0T(\psi)\] are locally integrable, it suffices to show that
\[\pi^*F_*\Phi^*(T(\phi)\boxtimes T(\psi))=(\pi\times \pi)^*(T(\phi)\boxtimes T(\psi))\]
holds outside of $\R^m\times \{0\} \cup \{0\}\times\R^n$. 
Noting that $(\pi\times\pi)^*=R^*\Phi_0^*$,  it suffices to verify that \begin{align*}R^*\Phi_0^*=\pi^*F_*\Phi^*\iff\Phi_0^*=F_0^*\pi^*F_*\Phi^*.\end{align*}
Since $\pi\circ F_0=F\circ\beta$, we find $F_0^*\pi^*=\beta^*F^*$ and we should verify
\[\Phi_0^*=\beta^*F^*F_*\Phi^*,\]
and since $F^*F_*=\id$, we are left with the equality 
\[\Phi_0^*=\beta^*\Phi^*,\]
which holds since $\Phi_0=\Phi\circ\beta$. 

Now assume $k>0, l=0$, so that $\psi=\chi$. This case then follows immediately from \cite[equation 2.1.13]{alesker_radon}. 
The remaining case $k=0,l=0$ amounts to the verification $\tau(\chi\boxtimes\chi)=\tau(\chi)\boxtimes\tau(\chi)$, which follows from $\delta_0\boxtimes\delta_0=\delta_0$.

\endproof

\subsection{Symmetry of closed vertical forms}

Closed, vertical forms on the sphere bundle play a key role in valuation theory. In the translation invariant case, such forms exhibit a certain symmetry, as we now proceed to show.

\begin{Lemma}\label{lemma:gray_symmetry}
	If $\tau \in \Omega^{n-k}(S^{n-1}, \largewedge^k (\RR^n)^*)$ is closed and vertical, then the form $\overline \tau\in   \Omega^{n-k}(S^{n-1}, \largewedge^{n-k} (\RR^n)^*)$, corresponding to $\tau$ under the Hodge star $\largewedge^k (\RR^n)^* \simeq 
	\largewedge^{n-k} (\RR^n)^*$, belongs to the subspace 
	$ C^\infty (S^{n-1}, \Sym^2 (\largewedge^{n-k} T^* S^{n-1}))$. 
\end{Lemma}

\begin{proof}
	Let $(x_1,\ldots, x_n,\xi_1,\ldots, \xi_n)$ be the standard coordinates on $\RR^n\times \RR^n \supset \RR^n\times S^{n-1}$. Let $R= \sum \xi_i \frac{\partial}{\partial x_i}$ denote the Reeb vector field. Since $\tau$ is vertical, one has $i_R \overline \tau=0$. Hence, using that $\tau$ is closed, one obtains 
	$$ 0 = d(i_R\overline \tau ) = \sum  d\xi_i \wedge i_{\frac{\partial}{\partial x_i}} \overline \tau.$$
	Hence  by \cite[Proposition~2.2]{Gray:SomeRelations} (see also \cite[Remark 5.8]{BFSW:Tube}) we deduce that 
	$\tau_\xi \in \Sym^2 (\largewedge^{n-k} T^*_\xi S^{n-1})$,  concluding the proof.
\end{proof}

Let $\tau \in\Omega^{n-k}_{-\infty}(V^*, \wedge^{k}V^*\otimes \ori(V))$.
Using the identification $\wedge ^k V^*\simeq \wedge^{n-k}V\otimes \wedge^n V^*$, we define the generalized form 
$\overline \tau\in\Omega^{n-k}_{-\infty}(V^*, \wedge^{n-k}V\otimes \Dens(V))$ corresponding to $\tau$. 
\begin{Proposition}\label{prop:symmetry}
	Let $\tau \in \Omega^{n-k}_{-\infty} ( V^*, \largewedge^k V^* \otimes \ori(V))$ be the $0$-homogeneous current of a valuation. Considered as an element of $C^{-\infty}(V^*, \largewedge^{n-k} V \otimes \largewedge^{n-k} V\otimes \Dens(V))$, $\overline \tau$ belongs to the subspace  
	$C^{-\infty}(V^*, \Sym^2(\largewedge^{n-k} V)\otimes \Dens(V))$.
\end{Proposition}

\begin{proof}
	If $k=0$, there is nothing to prove. Suppose therefore that $k>0$. 
	and use Lemma \ref{lem:zero_homogeneous} to write $\tau=r^0\tau'$, where $\tau'$ is vertical and closed.

	Assume first that $\tau'$ is a smooth form. Fix a Euclidean inner product on $V$. By Lemma~\ref{lemma:gray_symmetry} we know that $\overline\tau'$ is symmetric. It follows that $\overline \tau = \pi^* \overline \tau'$ is symmetric on $V\setminus\{0\}$. Since $\overline \tau$ is locally integrable, the symmetry of $\overline \tau$ follows. 
	
	In the general case, choose an approximate identity $\rho_j\in C^\infty(\mathrm{SL}(V^*))$, and set $\tau'_j=\rho_j\ast\tau'$. Then $\tau'_j$ is smooth, vertical and closed, and so $\rho_j\ast \overline\tau = r^0 ( \rho_j \ast \overline \tau')= r^0 \overline{\tau_j'}$ is symmetric. Taking $j\to\infty$ proves the proposition.
	
\end{proof}

\section{Fourier transform of differential forms}

\subsection{Fourier transform of smooth and generalized differential forms}
Let $V$ and $F$ be  finite-dimensional vector spaces of over the reals, and 
let $\Omega_{\mathcal S}^k(V, F)=\mathcal S(V, \largewedge^k V^* \otimes F)\subset \Omega^k(V,F)$ denote the Schwartz space of differential $k$-forms on $V$ with values in $F$ and rapidly decreasing coefficients.

\begin{Definition}\label{def:fourier_forms1}
	 The Fourier  transform $$\FC: \Omega_{\mathcal S}^k(V, F) \to \Omega_{\mathcal S}^{n-k}(V^*, \ori(V)\otimes F)$$  is defined as follows.  For $\omega \in \Omega_{\mathcal S}^k(V, F)$, 	
	the  Hodge star isomorphism \eqref{eq:ident}
	allows to consider $\omega$ as a map $\wt \omega:V\to  \largewedge^{n-k} V \otimes \Dens(V) \otimes \ori(V)\otimes F$.
	Hence we may define for $\xi\in V^*$
	$$ \FC(\omega)(\xi) =\int_{V} e^{2\pi \mathbf i \langle x,\xi\rangle}  \wt\omega(x) \in   \largewedge^{n-k} V\otimes \ori(V)\otimes F.$$ 
\end{Definition}

For $\omega\in \Omega^k(V, F)$ and $\eta\in \Omega_{c}^{n-k}(V, \ori(V) \otimes F^*)$, the 
isomorphism 
$$ \wedge^{k} V^* \otimes \wedge^{n-k} V^* \otimes \ori(V) \simeq \Dens(V)$$ 
allows to define the function $V\ni x\mapsto  \omega(x)\wedge  \eta(x)\in \Dens(V)$ and hence the pairing 
\begin{equation}\label{eq:canPairing} \langle \omega, \eta\rangle = \int_V  \omega(x)\wedge \eta(x) .\end{equation}
 Note that the inclusion of $\Omega^k(V, F)$ into $\Omega^k_{-\infty}(V, F)$ defined by this pairing coincides with the canonical inclusion of $C^\infty(M,\mathcal E)$ into $C^{-\infty}(M,\mathcal E)$.

\begin{Lemma}\label{lemma:Fsymmetry}
	Let $\omega\in  \Omega_{\mathcal S}^k(V, F)$ and $\eta \in \Omega_{\mathcal S}^k(V^*, F^*)$. 
	Then 
	$$ \langle \FC\omega, \eta\rangle = (-1)^{k(n-k)}\langle \omega, \FC\eta\rangle.$$
\end{Lemma}
\begin{proof}
This follows immediately from the fact that pointwise we have
$$  \wt \omega_x\wedge \eta_\xi = (-1)^{k(n-k)}  \omega_x\wedge \wt \eta_\xi\in \Dens(V)\otimes \Dens(V^*).$$   
\end{proof}

\begin{Definition} The Fourier transform of tempered generalized forms on $V$ with values in $F$ 
	$$ \FC \colon \Omega_{\mathcal S'}^{k}(V, F)  \to\Omega_{\mathcal S'}^{n-k}(V^*, F \otimes \ori(V))$$
	is defined by
	$$ \langle \FC T , \eta \rangle = (-1)^{k(n-k)} \langle T, \FC \eta \rangle,\quad \forall\eta \in \Omega_{\mathcal S}^k(V^*, F^*).$$
\end{Definition}

In view of Lemma~\ref{lemma:Fsymmetry}, this definition continuously extends the Fourier transform of smooth forms.

 Let us discuss functorial properties of the Fourier transform.
 
 \begin{Lemma}\label{lem:F_inversion}
The inversion formula 
$$ (\FC_{V^*} \times \id )\circ \FC_V (\omega)= (-1)^{kn} (-\id)^* \omega  $$ 
holds for all $\omega \in \Omega^k_{\mathcal S'}(V,F)$. 
\end{Lemma}
\begin{proof}
	This follows immediately from the classical inversion formula.
\end{proof}

	\begin{Lemma}\label{Fourier:funct}
		Let $i\colon V\to W$ be a monomorphism, and $p:W\to V$ an epimorphism. Write $n=\dim V$, $m=\dim W$. Then 
		$$ \FC_W\circ i_*  =  (-1)^{(m-n)(n-k)}(i^\vee)^* \circ \FC_V\colon  \Omega_{\mathcal S'}^{k}(V, F) \to  \Omega_{\mathcal S'}^{ n-k}(W^*, F \otimes \ori(V)), $$
		and 
		\[(p^\vee)_*\circ\FC_V=\FC_W\circ p^*\colon  \Omega_{\mathcal S'}^{k}(V, F) \to  \Omega_{\mathcal S'}^{ m-k}(W^*, F \otimes \ori(V)).\]
	\end{Lemma}

\begin{proof} We may assume $F=\R$. It suffices to show that 
	\begin{equation}\label{eq:pullback}i^*\FC_{W^*}\varphi =  (-1)^{(m-n)(n-k)} \FC_{V^*}(i^\vee)_*\varphi\end{equation}
	for every  Schwartz test function	$\varphi\in \Omega_{\mathcal S}^{k+m-n}(W^*, \ori(V)\otimes \ori(W) )  $. Indeed, if this holds, then for every  $T\in  \Omega_{\mathcal S'}^{k}(V)$ one has
	\begin{align*}  \langle \FC(i_* T), \varphi\rangle & = (-1)^{(k+m-n)(n -k)}  \langle i_* T, \FC\varphi\rangle\\
		& =  (-1)^{(k+m-n)(n -k)+(m-n+k)(n-k)+k(n-k)}  \langle T, i^* \FC\varphi\rangle\\
		& =  (-1)^{k(n-k)+(m-n)(n-k)}  \langle T,  \FC((i^\vee)_*\varphi)\rangle\\
		& =   (-1)^{(m-n)(n-k)}\langle \FC T,  (i^\vee)_*\varphi\rangle\\
		& =   (-1)^{(m-n)(n-k)}\langle (i^\vee)^*\FC T, \varphi\rangle
	\end{align*}
	For the proof of \eqref{eq:pullback} choose linear coordinates $(y_1,\ldots, y_{m})$ on $W$ such that 
	$ i(V)= \{ y_{n+1} = \cdots = y_{m}=0\}$ and let $x_j= i^*y_j$ for $j=1,\ldots, n$. Let $(\eta_1,\ldots, \eta_{m})$ and $(\xi_1,\ldots,\xi_{n})$ be dual coordinates on $W^*$ and $V^*$. For a multi-index $I=(i_1<\dots<i_r)$ we write $dy_I=dy_{i_1}\wedge\dots\wedge dy_{i_r}$, and similarly for other coordinates. 
	
	We may assume that $\varphi= f(\eta) d\eta_I\wedge  d\eta_{J} \otimes \sigma_W\otimes \sigma_V$ where $f$ is a Schwartz function, $I=(1, \dots, k)$, $J=(n+1, \dots, m)$. 
	
	Then, writing $d\eta=d\eta_1\wedge\dots\wedge d\eta_{m}$, we have
	$\ast (d\eta_I \wedge d\eta_{J})=(-1)^{(m-n)(n-k)}dy_{I^c}$, where $I^c=(k+1, \dots, n)$.
	Thus
	$$ i^* \circ \FC_{W^*} \varphi = (-1)^{(m-n)(n-k)}\left( \int_{W^*} e^{2\pi\mathbf i\langle \eta, i(x)\rangle} f(\eta)\;  d\eta \otimes \sigma_W \right) dx_{I^c} \otimes \sigma_V.$$
	
	Before we treat $\FC_{V^*}(i^\vee)_*\varphi$, recall that the defining property of fiber integration along an epimorphism $p\colon W\to V$ is
	$ \int_W p^* \omega \wedge  \psi = \int_V \omega \wedge  p_* \psi$, where $\omega$ and $\psi$ are Schwartz test forms. Thus, if $\psi\in \Omega_{\calS}^k(W,\ori(W))$, say $\psi=f(y) dy_I\wedge dy_J \otimes \sigma_V$ where $p(y)=(y_1,\dots, y_{n})$ and $dy_J=dy_{n+1}\wedge \cdots \wedge  dy_{m}$, then
	$$ p_* \psi = \left( \int_{y\in p^{-1}(x)}  f(y) \; dy_J\otimes \sigma_{\ker p} \right) dx_I \otimes \sigma_V,$$
	where $\sigma_V \otimes \sigma_{\ker p}  = \sigma_W$.   
	
	Therefore, writing $i^\vee =q$ we obtain
$$ (i^\vee)_*\varphi = \left( \int_{\eta \in q^{-1}(\xi)}  f(\eta)  d\eta_J \otimes \sigma_{\ker q} \right) d\xi_{I}.$$
Hence, putting $d\xi=d\xi_1\wedge\dots\wedge d\xi_{n}$, we find
\begin{align*}  \FC\circ (i^\vee)_*\varphi& = \left(\int_{V^*}  \left( \int_{\eta \in q^{-1}(\xi)}  f(\eta)  d\eta_J \otimes \sigma_{\ker q}\right) e^{2\pi \mathbf{i} \langle \xi,x\rangle}  d\xi  \otimes \sigma_V  \right)  dx_{I^c} \otimes  \sigma_V\\
& =\left( \int_{W^*} e^{2\pi\mathbf i\langle \eta, i(x)\rangle} f(\eta)\;  d\eta\otimes \sigma_W \right) dx_{I^c} \otimes \sigma_V,
\end{align*}
as claimed.	

Combined with the inversion formula, we deduce from the first identity 
 for $S\in \Omega^{k}_{\calS'}(V^*)$ that 
	\begin{align*} (-1)^{(m-k)m} (-\id)^* i_*  \FC_{V^*} S &= \FC_{W^*}\FC_Wi_* \FC_{V^*} S=(-1)^{(m-n)k}\FC_{W^*}(i^\vee)^*\FC_V \FC_{V^*} S\\
	&= (-1)^{kn+(m-n)k}\FC_{W^*}(i^\vee)^* (-\id)^* S\\&=(-1)^{mk}\FC_{W^*}(i^\vee)^* (-\id)^* S.
	\end{align*} 
Recall that $i_*\FC _{V^*}S\in \Omega_{\mathcal S'}(W^*, \ori(W))$. Accounting for the action of $(-\id)^*$ on $\ori(W)$, which is given by $(-1)^m$, while the action of $(-\id)^*$ appearing in the inversion formula is only on the form and not on its values, we have $(-\id)^* i_*  \FC_{V^*} S= (-1)^m   i_*  \FC_{V^*} (-\id)^*S$. This concludes the proof of the second identity.
\end{proof}

\begin{Corollary}\label{cor:homogeneous_fourier}
	If $\omega\in \Omega_{\mathcal S'}^{k}(V, F)$ is $r$-homogeneous, then $\FC\omega\in\Omega_{\mathcal S'}^{n-k}(V^*, F\otimes \ori(V))$ is $(-r)$-homogeneous. 
\end{Corollary}

\begin{Lemma}\label{lem:dif_Euler_interchange}
	For a tempered generalized form $\omega\in\Omega^k_{\mathcal S'}(V)$ one has $$\FC(d\omega)=(-1)^{ k+1}  2\pi\mathbf{i}\cdot i_E(\FC\omega),\qquad  d(\FC\omega)= (-1)^{ k+1} 2\pi\mathbf{i}\cdot \FC(i_E\omega).$$
\end{Lemma}
\begin{proof}
	This follows immediately from basic properties of the Fourier transform: $\FC(\partial_{j}f)=  -2\pi\mathbf i\xi_j\FC f$, $\partial_j(\FC f)=2\pi \mathbf i \FC(x_j f)$. 
\end{proof}

\subsection{Fourier transform of valuation-type differential forms}

The goal of this section is to show that valuation-type differential forms are closed under the Fourier transform.

 First, we consider closed forms.
\begin{Proposition}\label{prop:fourier_interchange2}
Let $i:W\hookrightarrow V$ be a monomorphism, and let $p\colon V\to W$ be an epimorphism. Put $n=\dim V$, $m=\dim W$. Let $\omega\in\Omega^k_{-\infty}(V)$   be closed, $0$-homogeneous, smooth outside the origin and satisfying $i_E\omega=0$. Then  $\FC \omega$ has the same properties, and the identities 
	 $$ \FC_W i^*\omega=(i^\vee)_*\FC_V\omega\quad \text{and}\quad   (p^\vee)^*\FC_{V}\omega=   (-1)^{(n-m)(n-k)}\FC_{W} p_*\omega$$ 
 hold.
\end{Proposition}
\begin{proof}
 Observe first that $\FC_V\omega$ is closed, $0$-homogeneous, smooth outside the origin, and satisfies $i_E\FC_V\omega=0$. Indeed, it is $0$-homogeneous by Corollary \ref{cor:homogeneous_fourier}, and it is smooth outside of the origin by \cite[Theorem 3.2.4]{lemoine_homogeneous}. We have $d\FC_V\omega=0$ and $i_E\FC_V\omega=0$ by Lemma \ref{lem:dif_Euler_interchange}. In particular, both sides of the asserted equalities are well-defined.
	
	Fix a Euclidean structure on $V$. For a compactly supported form $\eta\in \Omega^{k}(W) $ we must show
	\[ \langle \FC_W i^*\omega, \eta\rangle=\langle (i^\vee)_*\FC_V\omega, \eta\rangle.\] 
	
	Choose an approximate identity $\mu_\epsilon(x)dx\in\mathcal M^\infty(\R^n)$ with $\mu_\epsilon$ Schwartz. Using Lemma \ref{Fourier:funct} we find 
	
	\begin{align*}
	\langle \FC i^*\omega, \eta\rangle&=(-1)^{k(m-k)}	\langle  i^*\omega,  \FC\eta\rangle=(-1)^{k(m-k)}\lim_{\epsilon\to 0} \langle i^*(\omega\ast\mu_\epsilon), \FC \eta\rangle\\&
	=(-1)^{k(m-k)}\lim_{\epsilon\to 0} \langle \omega\ast\mu_\epsilon, i_*\FC \eta\rangle
	=(-1)^{k(m-k)}\lim_{\epsilon\to 0}\langle \omega, \FC (i^\vee)^*\eta \ast \mu_\epsilon\rangle \\&=(-1)^{k(m-k)}\lim_{\epsilon\to 0}\langle \omega, \FC ( (i^\vee)^*\eta \cdot \FC^{-1}\mu_\epsilon)\rangle
\\&=(-1)^{ k(m-k)+k(n-k)}\lim_{\epsilon\to 0}\langle  \FC \omega, (i^\vee)^*\eta \cdot \FC^{-1}\mu_\epsilon\rangle \\&= 	(-1)^{ k(m-k)+k(n-k)+k(m-k)+k(n-k)}\langle   (i^\vee)_*(\FC \omega), \eta\rangle, 
\end{align*}
concluding the proof of the first identity. The second follows from first via the inversion formula. 
\end{proof}

\subsubsection{Differential forms related  to valuations}

\label{key}
We now focus our attention  on the forms in 
$$ \mathcal S'_{n-k,k}(V^*)=\Omega_{\mathcal S'}^{n-k}(V^*, \wedge^k V^*\otimes \ori(V))
$$
which play an important role in valuation theory.
Given    $\omega\in \mathcal S'_{n-k,k}(V^*)$, we can  consider its Fourier transform
$$ \FC\omega \in \Omega^k_{\mathcal S'}(V, \wedge ^k V^*)
.$$

\begin{Definition}
	Given $\omega\in \mathcal S'_{n-k,k}(V^*)$, we define its Fourier transform
	$$\FC^0\omega\in \mathcal  S'_{k,n-k}(V)\otimes  \Dens(V),$$
	as the Fourier transform $\FC\omega$ of Definition \ref{def:fourier_forms1}, combined with the  Hodge star $ *_{V^*} \colon \largewedge^{k} V \to \largewedge^{n-k} V^* \otimes \dens(V^*)\otimes \ori(V)$ according to our convention \eqref{eq:changingF}.
\end{Definition}

The inversion formula for $\FC^0$ assumes the following form.

\begin{Lemma}\label{lem:inversion_F0} It holds for all $\omega \in \mathcal S'_{n-k, k}(V^*)$ that 
\[ (\FC^0_{V} \times \id )\circ \FC^0_{V^*} (\omega)= (-\id)^* \omega. \]
\end{Lemma}
\proof
It follows from Lemma \ref{lem:F_inversion} that $(\FC_{V} \times \id )\circ \FC_{V^*} (\omega)=(-1)^{n(n-k)} (-\id)^*$, while the composition of the  Hodge star with itself contributes a factor of $(-1)^{k(n-k)}$. It remains to note that $(-\id)^*$ acts also on the space $\wedge^k V^*\otimes\ori(V)$ by the factor $(-1)^{k+n}$, and $(-1)^{n(n-k)} (-1)^{k(n-k)}(-1)^{k+n}=1$.
\endproof

Furthermore, $\FC^0$ is compatible with exterior products. 

\begin{Lemma}\label{lem:extProduct}   If $V,W$ are vector spaces of dimensions $n, m$, then

		$$ \FC^0(\omega \boxtimes \zeta) = (-1)^{km+ ln}\FC^0 \omega \boxtimes \FC^0 \zeta$$ 
		holds for all  $\omega\in \mathcal S'_{n-k,k}(V^*)$ and $\zeta\in \mathcal S'_{m-l,l}(W^*)$
\end{Lemma}
\begin{proof}

	First observe that if $F_1,F_2$ are vector spaces and 
	\begin{gather*}
		\FC_V\colon \Omega_{\mathcal S'}^{n-k} (V^*, F_1)\to \Omega_{\mathcal S'}^k (V, or (V) \otimes F_1)\\ 
				\FC_W\colon \Omega_{\mathcal S'}^{m-l}(W^*, F_2)\to\Omega_{\mathcal S'}^l (W, or (W) \otimes F_2)\\ 
				\FC_{V\times W}\colon\Omega_{\mathcal S'}^{n-k,m-l} (V^*\times W^*,  F_1\otimes F_2)\to \Omega_{\mathcal S'}^{k,l} (V\times W, or (V\times W) \otimes F_1\otimes F_2), 
	\end{gather*}
	denote the corresponding Fourier transforms, then 
	$$ \FC_{V\times W} =(-1)^{k(m-l)} \FC_V \boxtimes \FC_W.$$
	Let $I_V$, $I_W$, $I_{V\times W}$ be the respective Hodge star isomorphisms as in eq. \eqref{eq:ident}. Then 
	$$ I_{V\times W} =  (-1)^{l(n-k)} I_V\boxtimes I_W.$$
	This completes the proof. 
\end{proof}

\subsubsection{Existence of the Fourier transform on valuations}
We are now ready to prove the main result of this section.

\begin{Proposition}\label{prop:existenceFourier} Assume that $\omega\in \Omega^{n-k}_{-\infty}(V^*, \wedge ^k V^*\otimes \ori(V))$ is  valuation-type, i.e.\ it has the following properties
	\begin{enumerate}
		\item $0$-homogeneous.
		\item $i_E\omega=0$.
		\item vertical: $\omega\wedge E=0$.
		\item closed: $d\omega=0$. 
	\end{enumerate}
	Then $\FC^0\omega \in \Omega^k_{-\infty}(V,\wedge  ^{n-k} V\otimes \ori(V^*)\otimes \Dens(V))$ is  valuation-type.
	
	Furthermore, if $\omega$ is smooth outside of the origin, then so is $\FC^0\omega$.
\end{Proposition}
\begin{proof}
	
		The case of $k=0$ is straightforward. Since $\omega\wedge E=0$, it follows that $\omega$ is supported at the origin. As it is $0$-homogeneous, it must be a multiple of the delta measure at the origin. Then $\FC^0\omega$ is constant on $V$, and so trivially   valuation-type. 
		
		Assume $k\geq 1$. Using the identification $\wedge ^k V^*=\wedge^{n-k}V\otimes \wedge^n V^*$, we define the generalized form 
		$\overline \omega\in\Omega^{n-k}_{-\infty}(V^*, \wedge^{n-k}V\otimes \Dens(V))$ corresponding to $\omega$. 
			It holds that $d\overline\omega=0$. 
		
		Now $$\eta:=\FC^0\omega=\FC\overline \omega\in\Omega^k_{-\infty}(V, \wedge^{n-k}V\otimes \Dens(V)\otimes \ori(V^*))$$ is $0$-homogeneous by Corollary \ref{cor:homogeneous_fourier}. The identification $\wedge^k V^*= \wedge^{n-k}V\otimes \Dens(V)\otimes \ori(V^*)$ gives a corresponding $0$-homogeneous generalized form $\overline \eta\in \Omega^k_{-\infty}(V, \wedge^{k}V^*)$. 
		
		We have by Lemma \ref{lem:dif_Euler_interchange} $$d\eta=d\FC\overline\omega= (-1)^{n-k+1}2\pi\mathbf i\FC( i_E\overline\omega)=0,$$ that is $\eta$ is closed, and hence so is $\overline\eta$. Similarly, $i_E\overline\eta=i_E\FC\omega=(-1)^{n-k}\frac{1}{2\pi \mathbf i}\FC d\omega=0$.

		The symmetry of $\overline\omega$ established in Proposition~\ref{prop:symmetry} immediately implies the symmetry of $\overline\eta$. Combined with the identity $i_E\overline\eta=0$, it follows that $\eta \wedge E=0$.
		
		Finally, if $\omega$ is smooth on $V^*\setminus\{0\}$ then so is $\overline\omega$, and by \cite[Theorem 3.2.4]{lemoine_homogeneous}, $\eta$ is smooth on $V\setminus\{0\}$. 
	
\end{proof}

\section{Alesker--Fourier transform of translation-invariant valuations}

Proposition \ref{prop:existenceFourier} allows to define the $\FC$-transform of valuations as in Definition~\ref{def:FT}. More precisely:
\begin{Definition}
	Let $V$ be an $n$-dimensional linear space, and $0\leq k\leq n$. The $\FC$-transform of a valuation $\phi\in\Val^{-\infty}_k(V)$ is the unique valuation $\FC\phi\in\Val^{-\infty}_{n-k}(V^*)$ such that 
	$\tau(\FC\phi)=\FC^0\tau(\phi)$.
\end{Definition}

It follows from Proposition \ref{prop:existenceFourier} that $\FC \phi$ is a smooth valuation whenever $\phi$ is a smooth valuation.

We will  prove  below  that the $\FC$-transform of valuations thus defined coincides with the Fourier transform constructed by Alesker in \cite{Alesker:Fourier}.

\begin{Lemma} \label{lemma:pairings}For $i=1,2$ consider the pairings  $$P_i\colon \Val^\infty(V)\otimes \Dens(V^*)\otimes \Dens(V)\times \Val^\infty(V)\otimes \Dens(V^*)\to \C$$
	 defined by
	$$P_1(\phi \otimes \alpha\otimes v, \psi \otimes \beta)= \langle \phi \otimes \alpha *\psi \otimes \beta(\{0\}), v\rangle$$ 
	and 
	 $$P_2(\phi \otimes \alpha\otimes v , \psi \otimes \beta)= \langle \alpha,v\rangle \langle  ((-\id)^*\phi \cdot \psi)_n, \beta\rangle .$$ 
	Then $P_1=P_2$.
\end{Lemma}
\begin{proof}
 Choose a Euclidean inner product to  identify $V$ with $\RR^n$ and $\Dens(V)\simeq \CC \simeq \Dens(V^*)$. 
Choose $S\subset \RR^n$ with volume $1$. Then 
\begin{align*} \phi*\psi(\{0\})& = a_*(\phi\boxtimes \psi)(\{0\})\\
	& = \frac{1}{n!}\left.\frac{\partial^n}{\partial t^n}\right|_{t=0} \phi\boxtimes \psi((-\id,\id)\circ\Delta( tS))\\
	& =  (\Delta^* ((-\id)^*\phi\boxtimes \psi)_n\\
	&= ((-\id)^*\phi\cdot \psi)_n.
	\end{align*}
\end{proof}

	We are now ready to prove Theorem~\ref{thm:properties}, which we restate below.
\begin{Theorem}  The $\FC$-transform of valuations has the following properties:
\begin{enumerate}
	\item $\FC$ commutes with the natural action of $\mathrm{GL}(V)$. 
		\item \label{F:inversion} Inversion formula: $(\FC_{V^*}\times \id)\circ \FC_V\phi =(-\id)^*\phi$.
		\item  Let $i:V\hookrightarrow W$ be an injective linear map, and $p=i^\vee:W^*\to V^*$. Then for $\phi\in \Val^\infty(W)$, $\FC_V( i^* \phi) = p_* (\FC_W\phi)$.
		\item $\FC (\phi \boxtimes \psi)= \FC\phi \boxtimes \FC\psi$ for  $\phi\in \Val^\infty(V)$, $\psi\in \Val^{\infty}(W)$.
		\item \label{F:hom} $\FC$ intertwines product and convolution: for $\phi,\psi \in\Val^\infty(V)$, $\FC (\phi\cdot \psi)=\FC\phi\ast\FC\psi$. 
		\item  Self-adjointness: $\langle \FC u,  \theta \rangle = \langle u, \FC \theta\rangle  $ for $u \in \Val^{-\infty}(V)$, $\theta\in \Val^\infty(V^*)$.
		\item \label{F:surjection} Let $p:V\to W$ be a surjective linear map, and $i=p^\vee:W^*\to V^*$. Then for $\psi\in\Val^{-\infty}(W)$, $\FC_V  ( p^* \psi)= i_*(\FC_W\psi)$.  
	\end{enumerate}
\end{Theorem}

\begin{proof}
	\begin{enumerate}
			\item This is clear by construction.
		\item Follows at once from Lemma \ref{lem:inversion_F0}.
	
		\item Follows from Propositions ~\ref{prop:pullback_linear}, \ref{prop:pushforward_linear} and \ref{prop:fourier_interchange2}. The first two propositions assert that $i^* \phi$ corresponds to the pushforward under $p$ of the $0$-homogeneous current of the valuation, while $p_* (\FC_W\phi)$ is given by the restriction (using $i$) of the respective current. The last proposition asserts the equality of the two currents. 
		\item Follows from Proposition~\ref{prop:exterior_product} and Lemma~\ref{lem:extProduct}.

		\item Observe that $\Delta_V^\vee = a_{V^*}$, while $\FC(\phi \boxtimes \psi) =\FC\phi \boxtimes \FC\psi$ by (2). Let $\mu_j\in\mathcal M^\infty_c(\GL(V\times V))$ be an approximate identity such that $\mu_j\to \delta_{\id}$ as $j\to \infty$.  
		Denoting $\zeta=\phi\boxtimes\psi\in\Val^{-\infty}(V\times V)$, $\zeta_j=\zeta\ast \mu_j\in \Val^\infty(V\times V)$, we have
	
		\begin{align*}
			\FC(\phi\cdot\psi )&= \FC(\Delta^*(\zeta))=\lim_{j\to \infty} \FC(\Delta^*(\zeta_j))=\lim_{j\to \infty} a_*\FC(\zeta_j)\\&=\lim_{j\to \infty} a_*(\FC(\zeta)\ast \mu_j)=a_*(\FC(\zeta))=\FC(\phi)\ast\FC(\psi),
		\end{align*}
		
		where the second equality follows from the sequential continuity of $\Delta^*$ in the H\"ormander topology \cite{Duistermaat:FIO} of a subspace of valuations with restricted wavefront to which $\zeta$ belongs, and in which $\zeta_j$ converges, the third equality is property (3), the fourth equality follows from the equivariance of the $\FC$-transform on valuations with respect to the general linear group, and the fifth equality follows from the sequential continuity of $a_*$ in the H\"ormander topology, similarly to $\Delta^*$. 
	\item 
	
	  By continuity, it is enough to prove the claim for $u\in \Val^\infty(V)\subset \Val^{-\infty}(V)$.
		
Let $L_0\colon \Val^\infty(V^*)\otimes \Dens(V)\otimes \Dens(V^*)\to \C$ be defined by composition of  evaluation at a point with the identification $\Dens(V)\otimes \Dens(V^*)\simeq\C$. Similarly, let $L_n \colon \Val^\infty(V)\otimes \Dens(V^*) \to \C$ be the projection to the top-degree component, which is just $\Dens(V)$, followed by the same identification. 
Using $\FC\chi = \vol\otimes \vol^*$, the homomorphism property \eqref{F:hom}, Lemma~\ref{lemma:pairings}, and the inversion formula \eqref{F:inversion}, we obtain
	\begin{align*}
		\langle \FC_V u, \theta\rangle & = L_n(\FC_V u \cdot \theta)\\
		&  =  L_0 (\FC_{V^*}\times \id(\FC_V u \cdot \theta))\\
		& =  L_0((\FC_{V^*}\times \id) \circ \FC_V u  )	 * \FC_{V^*}\theta)\\
		& = P_1 ( (\FC_{V^*}\times \id) \circ \FC_V u , \FC_{V^*} \theta) \\
		& = P_2 ( (\FC_{V^*}\times \id) \circ \FC_V u , \FC_{V^*} \theta) \\
		& = L_n(u \cdot \FC_{V^*} \theta)\\
		& = \langle u, \FC_{V^*} \theta\rangle,
	\end{align*}
as claimed.

\item For smooth valuations, this follows from Propositions~\ref{prop:pullback_linear} and \ref{prop:pushforward_linear} and Lemma~\ref{Fourier:funct}. The general case then follows by continuity.
\end{enumerate}
\endproof

 It remains to show that our $\FC$- and Alesker's Fourier transforms coincide. Since this  verification will be reduced to the two-dimensional case,  we will use the following statement about valuations in the plane.

 \begin{Lemma}\label{lemma:plane}
	Suppose that $\phi\in \Val^{-\infty}_1(\RR^2)$ is a generalized valuation, and  $g(e^{i\theta}) d\theta$ a generalized measure such that
	$$ \phi(K)= \int_0^{2\pi} h_K(e^{\mathbf i \theta}) g(e^{i\theta}) d\theta$$
	for every smooth and strictly positively curved convex body $K$ in $\RR^2$. Then
	$$ \tau(\phi)= \frac{g(-y/|y|)}{|y|^3}   (y_1 dx_1 + y_2 dx_2) \otimes  ( y_1 dy_2- y_2dy_1),$$
	where $(x,y)$ are the coordinates on $\RR^2\times \RR^2 \supset \RR^2\times S^1$. 
\end{Lemma}
\begin{proof}
	If  $\psi = \int_{\nc(K)} \omega$, where $ \omega= f_1(y) dx_1 + f_2(y) dx_2 \in \Omega^1(S\RR^2)$, then a computation shows that 
	$$ \psi(K)= \int_{S^1} (-f_1(y) y_2 + f_2(y) y_1 \, )dS_1(K,y)=2 V(K,-f_1 y_2 + f_2 y_1 ) .$$
	
	The definition of $\phi$ implies
	$$ \phi\cdot \psi=  \int_{S^1} h(e^{\mathbf i \theta}) g(-e^{i\theta})d\theta$$
	 for  valuations of the form $\psi(K)= 2V(K,h)$. Hence 
	 $$T(\phi)=   (y_1 dx_1 + y_2 dx_2) \otimes g(-e^{i\theta}) d\theta,$$
	 from which the expression for $\tau(\phi)$ follows.
 \end{proof}

\proof[Proof of Theorem~\ref{thm:F=F}]  The Alesker construction of the Fourier transform strongly relies on the irreducibility theorem \cite{Alesker:Irreducibility}, and so we make use of it for the verification. Denote the Alesker--Fourier transform by $\FF$.
Observe that property (4) of Theorem \ref{thm:properties} holds for both definitions, while Alesker products of $k$-tuples of elements of $\Val_1^\infty(V)$ span a dense subset of $\Val_k^\infty(V)$ by the irreducibility theorem. It therefore suffices to verify that $\FC=\FF$ on $\Val_1^\infty(V)$.
Furthermore, since $\FC, \FF:\Val_1^{\pm, \infty} (V)\to \Val_{n-1}^{\pm, \infty} (V^*)\otimes\Dens(V)$ are both equivariant isomorphisms between irreducible $\GL_n(\R)$-modules, there exist 
constants  $c_{\pm}(V)$ such that $\FC=c_\pm(V)\FF$. Since by Theorem \ref{thm:properties} and \cite{Alesker:Fourier} both $\FC$ and $\FF$ intertwine restrictions and pushforwards by projections, the number $c_{\pm}(V)=c_{\pm}(n)$ depends only on the dimension of $V$. By the same token,  fixing an inclusion 
$i:\R^k\hookrightarrow \R^n$  shows that $c_\pm$ is independent of the dimension. Finally, the inversion formulas of Theorem \ref{thm:properties} and \cite{Alesker:Fourier} imply that $c_+^2=c_-^2=1$.

To determine the value of $c_+$, it suffices to note that on $\R^1$, $\FC(\vol_1)=\chi=\FF(\vol_1)$, that is $c_+=1$. 
\medskip

Now consider $\R^2=\C$, and define $\phi\in\Val_1(\R^2)$ by
\[\phi(K)=h_K(1)+h_K(e^{2\pi\mathbf i/3})+h_K(e^{4\pi \mathbf i/3}),\] which is a continuous, and therefore  generalized, translation-invariant valuation. Then $\psi(K):=\phi(K)-\phi(-K)$ lies in $\Val_1^{-,\infty}(\R^2)$. 

According to Lemma~\ref{lemma:plane} and 
 denoting $\omega= \sign(y_1)\boxtimes \delta_0(y_2)dy_2 $, the $0$-homogeneous current of $\phi$ is
\[ \tau(\psi)=dx_1 \otimes \omega  + (e^{2\pi \mathbf i/3})^*(dx_1 \otimes \omega )+(e^{-2\pi \mathbf i/3})^*(dx_1 \otimes \omega ). \]
We have
\[\FC^0(dx_1\otimes \omega)=d\xi_2 \otimes  \FC(\omega),\]  
where
\[ \FC( \omega) (\eta_1, \eta_2)=\frac{1}{\pi \mathbf i \eta_1}\boxtimes 1 \, d\eta_1.\]

As $\FC$ commutes with rotations, a straightforward computation now shows that 
\begin{align*}\tau(\FC\psi)=\frac{3}{\pi\mathbf i \eta_1(\eta_1^2-3\eta_2^2)}(\eta_1 d\xi_1+\eta_2 d\xi_2) \otimes (\eta_1 d\eta_2 -\eta_2 d\eta_1),\end{align*}
 where $(\xi,\eta)$ are the coordinates on $\RR^2\times \RR^2\supset \RR^2\times S^1$.

Let us recall from  \cite{alesker2013fourier} the description of the Alesker--Fourier transform in $\R^2=\mathbb C$.
For \[ \zeta(K)=\int_0^{2\pi}h_K(e^{\mathbf i\theta}) f(\theta) d\theta,\]
write $f=f_+ + f_-$, $f_-= f_-^{hol}+f_-^{anti}$, where $f_+, f_-$ are the even and odd parts of $f$ on the circle, 
$f_-^{hol}(\theta)=\sum_{n\equiv 1(2),n>0}\widehat f(n)e^{\mathbf in\theta}$ and $f_-^{anti}(\theta)=\sum_{n\equiv 1(2),n<0}\widehat f(n)e^{\mathbf in\theta}$ the 
holomorphic and anti-holomorphic parts of $f_-$, respectively. Then
 \[ \FF\zeta(K)=\int_0^{2\pi}h_K(e^{\mathbf i\theta}) \tilde f(\theta) d\theta, \]
 where 
 \begin{align*}\tilde f(\theta)&=f_+(\theta+\frac \pi 2)+f_-^{hol}(\theta+\frac \pi 2)-f_-^{anti}(\theta+\frac \pi 2)=\sum \mathbf i^{|n|}\widehat f(n)e^{\mathbf i n\theta}.\end{align*}

Now $\psi$ is given by $\psi(K)=\int_0^{2\pi}h_K(e^{\mathbf i\theta})f_\psi(\theta)d\theta$ with 
\[ f_\psi(\theta)=\frac{1}{2\pi}(\delta_0(\theta)+\delta_{2\pi/3}(\theta)+\delta_{4\pi/3}(\theta)-\delta_\pi(\theta)-\delta_{\pi/3}(\theta)-\delta_{5\pi/3}(\theta)).\]
The Fourier series of $\delta_0(\theta)$ is 
\[ \delta_0(\theta)=\sum_{n=-\infty}^\infty e^{\mathbf in\theta},\]
implying that \[f_\psi(\theta)=\frac{1}{2\pi}\sum_{n=-\infty}^\infty (3e^{3\mathbf i n\theta}-3e^{3\mathbf i n(\theta+\pi)}).\]

It follows that $\mathbb F\psi$ is given by 
\[\mathbb F\psi(K)=\int_0^{2\pi} h_K(e^{\mathbf i\theta})g_\psi(\theta)d\theta \]  with
\begin{align*} g_\psi(\theta)&= \frac{1}{2\pi}\sum_{n=-\infty}^\infty (3\mathbf i^{3|n|}e^{3\mathbf i n\theta}-3\mathbf i^{3|n|}e^{3\mathbf i n(\theta+\pi)})=\frac{3}{\pi\mathbf i \cos3\theta}.\end{align*}

 In view of Lemma~\ref{lemma:plane} it remains to notice that on the unit circle, $\eta_1(\eta_1^2-3\eta_2^2)=4\cos^3\theta-3\cos\theta=\cos3\theta$.

\end{proof}

	\subsection{Example: even valuations}
	The Fourier transform on even valuations, acting on the Klain section or the Crofton measure of a valuation, amounts to the action of the orthogonal complement on the Grassmannian. Let us recover this fact using our construction. 
	
	We use the standard Euclidean structure and orientation on $\R^n$. For $E\in\Gr_k(\R^n)$, denote $ \phi_E(K)=\vol(P_{E^\perp} K)$, where $P_{E^\perp}$ is the orthogonal projection onto $E^\perp$.
	Choose orthonormal coordinates on $\R^n$ such that $E=\{x_{k+1}=\dots=x_n=0\}$. In view of Lemma~\ref{lemma:tauProj}, equation \eqref{eq:tauProj}, we have 
	\[ \tau(\phi_E)=dx_1\wedge\dots\wedge dx_k \otimes \delta_0(y_{k+1})\cdots \delta_0(y_{n}) dy_{k+1}\wedge \dots\wedge dy_n.\] 
	
	As \[\FC(\delta_0(y_{k+1})\cdots \delta_0(y_{n})  dy_{k+1}\wedge \dots\wedge dy_n) =(-1)^{k(n-k)}\delta_0(y_{1})\cdots \delta_0(y_{k}) dy_{1}\wedge \dots\wedge dy_k,\]  
we find that
\begin{align*}\tau(\FC(\phi_E))&=\FC^0\tau(\phi_E)\\&=(-1)^{k(n-k)}\ast (dx_1\wedge\dots\wedge dx_k)\otimes \delta_0(y_{1})\cdots \delta_0(y_{k}) dy_{1}\wedge \dots\wedge dy_k
\\&=\tau(\phi_{E^\perp}).\end{align*}

Now for any valuation given by the Crofton formula $\phi(K)=\int_{\Gr_{k}(\R^n)}\phi_E(K)dm(E)$, we deduce by linearity and continuity of the Fourier transform that $\FC\phi(K)=\int_{\Gr_{k}(\R^n)}\phi_{E^\perp}(K)dm(E),$
that is $\FC\phi(K)=\int_{\Gr_{n-k}(\R^n)}\phi_{F}(K)dm'(F)$, where $m'=\perp_*(m)$.

	\subsection{Example: intrinsic volumes}
	For the purpose of illustration, let us  compute the Fourier transform of the intrinsic volumes $V_k$ directly from our definition.

	\begin{Lemma} \label{lemma:tauV} For $k=1,\ldots, n-1$, 
		$$ \tau(V_k)=    d i_E  \lambda_k,   $$ where 
		$$ \lambda_k = \frac{1}{k!(n-k)! \vol(S^{n-k-1})} \frac{1}{|y|^{n-k}} \sum_{\pi} \sign(\pi) dx_{\pi_1}\cdots dx_{\pi_k} dy_{\pi_{k+1}} \cdots dy_{\pi_n}.$$
	\end{Lemma}
	\begin{proof}
		Integration over the normal cycle of  
		$$\kappa_k =  c_{k,n} \sum_{\pi} \sign(\pi) y_{\pi_{1}} dx_{\pi_2}\cdots dx_{\pi_{k+1}} dy_{\pi_{k+2}} \cdots dy_{\pi_n},$$
		where $ c_{k,n} = \frac{1}{k! (n-k)! \omega_{n-k}}$ and the sum is over the permutation group,
		yields the $k$th intrinsic volume. Note that 
		$$ d\kappa_k=  c_{k,n} \sum_{\pi} \sign(\pi) dx_{\pi_1}\cdots dx_{\pi_k} dy_{\pi_{k+1}} \cdots dy_{\pi_n} = (-1)^{n-k} a^* D\kappa_k=T(V_k).$$
		We have to pull this form back to $\RR^n\setminus \{0\}$. 
		First observe that since $j\circ \pi(y)= y/|y|$, where $j\colon S^{n-1}\to \RR^n$  denotes the inclusion and $\pi$ is the radial projection, we have 
		$$ \pi^* j^* dy_i = d( y_i/|y|)= \frac{1}{|y|} dy_i -  \frac{y_i}{|y|} \rho$$
		where $\rho= |y|^{-2}\sum_{j=1}^n y_j dy_j=\frac{1}{|y|}d(|y|)$.
		Thus 
		\begin{align*} \pi^* j^* ( dy_{\pi_{k+1}} \cdots dy_{\pi_n}) &=\frac{1}{|y|^{n-k}}  dy_{\pi_{k+1}} \cdots dy_{\pi_n}
			- \frac{1}{|y|^{n-k}}\rho \wedge \sum_{j=1}^k y_{\pi_j} i_{\partial_{y_{\pi_j}}} dy_{\pi_{k+1}} \cdots dy_{\pi_n}\\
			& =\frac{1}{|y|^{n-k}}  dy_{\pi_{k+1}} \cdots dy_{\pi_n}
			- \frac{1}{|y|^{n-k}}\rho \wedge i_E dy_{\pi_{k+1}} \cdots dy_{\pi_n}\\
			&= \frac{1}{|y|^{n-k}} i_E (\rho \wedge dy_{\pi_{k+1}} \cdots dy_{\pi_n})\\
			& = -\frac{1}{n-k}  i_E d(|y|^{-(n-k)} dy_{\pi_{k+1}} \cdots dy_{\pi_n}).
		\end{align*}
		Since  $i_E  d\lambda_k  = -d i_E \lambda_k  $ the claim follows.
	\end{proof} 
	According to \cite[Lemma 3.6]{Koldobsky:FourierAnalysis}, 
$$ \FC(|x |^{-n+k})(\xi)=    \frac{\vol(S^{n-k-1})}{\vol(S^{k-1})} \frac{1}{|\xi|^{k}}$$
and hence
$$\FC^0 \lambda_k =   \lambda_{n-k}.$$
  	Using Lemmas~\ref{lemma:tauV} and \ref{lem:dif_Euler_interchange} we conclude 
 $$\FC^0\tau(V_k)=  \FC^0 (d i_E \lambda_k)= - i_Ed \FC^0\lambda_k= - i_Ed\lambda_{n-k} = d i_E \lambda_{n-k} = \tau(V_{n-k}),$$
 that is, $\FC V_k = V_{n-k}$.

\bibliographystyle{abbrv}
\bibliography{ref_papers,ref_books}

\end{document}